\documentclass[12pt]{amsart}
\usepackage{amsfonts,latexsym,rawfonts,amsmath,amssymb,amsthm,mathrsfs}
\usepackage[pagewise]{lineno}
\usepackage[colorlinks=true,urlcolor=blue,linkcolor=blue,anchorcolor=blue,citecolor=blue]{hyperref}
\usepackage{enumitem}

\allowdisplaybreaks[4]
\usepackage{cite}

\makeatletter
\@namedef{subjclassname@2020}{\textup{2020} Mathematics Subject Classification}
\makeatother

\numberwithin{equation}{section}

\RequirePackage{color}
\theoremstyle{plain}

\hypersetup{hypertex=true,
            colorlinks=true,
            linkcolor=blue,
            anchorcolor=blue,
            citecolor=blue}

\def\p{\partial}
\def\n{\nabla}
\def\l{\langle}
\def\r{\rangle}
\def\R{\mathbb{R}}
\def\s{\mathbb{S}}
\date{}

\newtheorem{theorem}{Theorem}[section]
\newtheorem{lemma}{Lemma}[section]
\newtheorem{corollary}{Corollary}[section]

\newtheorem{remark}{Remark}[section]

\begin{document}

\title{Long time behavior of a class of non-homogeneous anisotropic fully nonlinear curvature flows}

\author{Weimin Sheng}
\address{Weimin Sheng: School of Mathematical Sciences, Zhejiang University, Hangzhou 310058, China.}
\email{weimins@zju.edu.cn}

\author{Jiazhuo Yang}
\address{Jiazhuo Yang: School of Mathematical Sciences, Zhejiang University, Hangzhou 310058, China.}
\email{yangjiazhuo@zju.edu.cn}



\maketitle
\begin{abstract}
    In this paper, we study a class of non-homogeneous anisotropic fully nonlinear curvature flows in $\R^{n+1}$. More precisely, we consider a hypersurface $M$ in $\R^{n+1}$ deformed by a flow along its unit normal with its speed $f(r)\sigma_k^\alpha$ where $\sigma_k$ is the $k$-th elementary symmetric polynomial of $M$'s principle curvatures, $r$ is the distance of the point on $M$ to the origin, $f$ is a smooth nonnegative function on $[0,\infty)$ and $\alpha > 0$. Under some suitable conditions on $f$, we prove that starting from a star-shaped and $k$-convex hypersurface, the flow exists for all time and converges smoothly to a sphere after normalization. In particular, we generalize the results in \cite{li2022asymptotic}.
\end{abstract}

\tableofcontents

\section{Introduction}
 Extrinsic curvature flows have been widely studied in the past few decades and many fascinating results have been produced. The typical examples are mean curvature flow  and Gauss curvature flow. Huisken proved in \cite{huisken1984flow} that along the mean curvature flow, a convex hypersurface in Euclidean space converges to a round point. That is, after appropriate rescaling, the convex hypersurface converges to a sphere. This result was later generalized to a large class of flows whose velocity is a homogeneous of degree one function of the principal curvatures. The examples include \cite{andrews1994contraction, chow1985deforming, chow1987deforming}, etc.. For homogeneous curvature functions with degree not equal to one, the situation is more different and subtle. For Gauss curvature flow, a convex hypersurface in $\R^{n+1}$ evolves by powers of Gauss curvature $K^\alpha$. A series of works \cite{firey1974shapes,andrews1999gauss,andrews2016flow,brendle2017asymptotic,guan2017entropy} have established that the limiting shape is a sphere when $\alpha > \frac{1}{n+2}$, while the limiting shape is an ellipsoid when $\alpha = \frac{1}{n+2}$.

As a natural extension, anisotropic geometric flows have also been widely studied. In particular, flows with speed involving radial distance have extensive applications in the subject of convex geometry. In \cite{li2020flow} Li, Sheng and Wang considered the following flow $X: M \times [0,T) \rightarrow \mathbb{R}^{n+1}$  to study the dual Minkowski problems
\begin{equation}\label{100}
    \begin{cases}
    \frac{\p X}{\p t} = -f(\nu)r^{\alpha}K\nu\\
    X(\cdot,0) = X_0
\end{cases}
\end{equation}
where $X_0: M \rightarrow \R^{n+1}$ is a star-shaped closed hypersurface of $\R^{n+1}$ which includes the origin inside, $K$ is the Gauss curvature of the hypersurface $M_t=X_t(M)$ we consider, $f$ is a given positive smooth functions on $\s^n$, $r$ is the distance from the point of the hypersurface $M_t=X_t(M)$
to the origin and $\nu$ is the unit normal. The dual Minkowski problem was proposed in \cite{huang2016geometric}(see also \cite{lutwak2018lp}), and is equivalent to finding convex bodies $\Omega$ whose boundary $\p \Omega$ satisfies the following equation
\begin{equation}\label{101}
    r^{n+1-q}Kf = u,
\end{equation}
where $u$ is the support function of $\p\Omega$ and $f$ is a given positive smooth functions on $\s^n$. Under some appropriate assumptions on $\alpha$ and $f$, the authors proved that the limiting shape of a convex hypersurface $X_0(M)=M_0$ under the flow \eqref{100} is a solution to the elliptic equation \eqref{101}.

It is more general to consider the following problems 
\begin{equation}\label{102}
    r^{n+1-q}\sigma_kf = u
\end{equation}
where $\sigma_k$ is the $k$-th elementary symmetric polynomial of principle curvatures. Guan,Lin and Ma \cite{guan2009existence} considered this problem when $q = 0$. In \cite{li2020asymptotic} Li, Sheng and Wang introduced the following flow
\begin{equation}\label{104}
    \begin{cases}
        \frac{\p X}{\p t} = -r^\alpha \sigma_k \nu,\\
        X(\cdot,0) = X_0.
    \end{cases}
\end{equation}
The authors proved that a convex hypersurface converges to a round point at infinity along \eqref{104} when $\alpha \geq k+1$. Furthermore, when $f$ is a constant, the equation \eqref{102} characterizes the self-similar solutions to the flow \eqref{104} with $\alpha = n+1-q$.

Denote 
$$\Gamma_k^+ = \{(\kappa_1,\cdots,\kappa_n)\in \R^n| \sigma_1(\kappa) > 0,\cdots,\sigma_k(\kappa) > 0\}.$$
We say a hypersurface $M$ is $k$-convex if $\kappa \in \Gamma_k^+$ everywhere on $M$.

Li et al. \cite{li2022asymptotic} studied the flow
\begin{equation}\label{105}
    \begin{cases}
        \frac{\p X}{\p t} = -r^\frac{\alpha}{\beta} \sigma_k^{\frac{1}{\beta}} \nu,\\
        X(\cdot,0) = X_0,
    \end{cases}
\end{equation}
and generalize the results \cite{li2020asymptotic} to star-shaped and $k$-convex hypersurfaces for some suitable constants $\alpha$ and $\beta$.

In this paper, we consider a class of anisotropic flows when the velocity function is inhomogeneous with respect to $r$. Let $f:[0,\infty) \rightarrow [0,\infty)$ be a smooth function such that $f(0) = 0$ and $f(r) > 0$ for $r > 0$. Given a star-shaped, $k$-convex and closed hypersurface $X_0: M \rightarrow \R^{n+1}$ which includes the origin inside, we consider the flow $X: M \times [0,T) \rightarrow \mathbb{R}^{n+1}$ satisfying 
    \begin{equation}\label{1}
         \begin{cases}
        \frac{ \p X}{ \p t} = -f(r)\sigma^\alpha_k\nu\\
        X(\cdot, 0) = X_0
    \end{cases}
    \end{equation}
where $\sigma_k$ is the $k$-th elementary symmetric polynomial of the principle curvatures, $r$ is the radial distance, $\alpha$ is some positive constant and $\nu$ is the outer normal vector field.

\begin{theorem}\label{thm1}
    Denote $g(r) = f(r) - r^{1+k\alpha}$ satisfying 
   \begin{enumerate}[label=(\roman*)]
       \item $g \geq 0$ and $g \equiv 0$ on $[0,\epsilon]$ for some small $\epsilon > 0$;
       \item $(1+k\alpha)\frac{g(r)}{r} \leq g'(r), \forall r > 0$.
   \end{enumerate} 
Let $X_0$ be a smooth, star-shaped and $k$-convex hypersurface in $\R^{n+1}$ which includes the origin $0$. If $k=1,\alpha > 0$ or $k \geq 2,\alpha \in \{\frac{1}{k}\} \cup [1,\infty)$, then the flow (\ref{1}) has a unique smooth, 
star-shaped and $k$-convex solution $X_t$ for all time $t > 0$, which converges to the origin. After a rescaling $\tilde{X}(\cdot,t) = e^{\gamma t}X(\cdot,t),\gamma = (C^k_n)^\alpha$, the hypersurface $\tilde{M}_t = \tilde{X}_t(M)$ converges exponentially to a sphere centered at the origin in the $C^\infty$ topology.
\end{theorem}
\begin{remark}
    There is a large class of functions $g$ satisfying our conditions in Theorem \ref{thm1}, for example
    $$g(r) = \begin{cases}
        0,0 \leq r \leq \epsilon,\\
        r^{1+k\alpha}e^{-\frac{1}{(r-\epsilon)^p}}, r > \epsilon
    \end{cases}$$
    where $p > 0$.
\end{remark}
\begin{theorem}\label{thm2}
     Denote $g(r) = f(r) - r^{\beta},\beta > 1 + k\alpha$ satisfying
      \begin{enumerate}[label=(\roman*)]
       \item $g \geq 0$ and $g'(0) = \cdots =g^{([\beta])}(0) = 0$;
       \item $(1+k\alpha)\frac{g(r)}{r} \leq g'(r), \forall r > 0$.
   \end{enumerate} 
Let $X_0$ be a smooth, star-shaped and $k$-convex hypersurface in $\R^{n+1}$ which includes the origin $0$. If $k=1,\alpha > 0$ or $k \geq 2,\alpha \in \{\frac{1}{k}\} \cup [1,\infty)$, then the flow (\ref{1}) has a unique smooth, star-shaped and $k$-convex solution $X_t$ for all time $t > 0$, which converges to the origin. After a rescaling $\tilde{X}(\cdot,t) = (1+(\beta-k\alpha-1)\gamma t)^{\frac{1}{\beta - k\alpha - 1}}X(\cdot,t),\gamma = (C^k_n)^\alpha$, the hypersurface
 $\tilde{M}_t= \tilde{X}_t(M)$ converges exponentially to a sphere centered at the origin in the $C^\infty$ topology.
\end{theorem}
\begin{remark}
    There is a large class of functions $g$ satisfying the conditions in Theorem \ref{thm2}, for example
    $$g(r) = r^{1+k\alpha}e^{-\frac{1}{r^p}},p>0;$$
    or
    $$g(r) = r^{l},l \geq [\beta]+1.$$
\end{remark}
This paper is organized as follows. In Section 2, we give the notions of the normalized flow \eqref{2}, compute the evolution equations of some basic geometric quantities along the normalized flow and reduce the flow to a parabolic equation on $\s^n$. In Section 3, we derive the $C^0$ estimate. As a consequence, we obtain some useful bounds for quantities appear in the evolution equations and show that $k$-convexity is preserved along the flow. In Section 4, we prove $C^1$ estimate and show that the hypersurface preserves star-shaped along the flow. In Section 5, we obtain the $C^2$ estimate along the flow and hence we have the long time existence of the flow. In Section 6, we show the normalized flow converges to a sphere in the $C^\infty$ topology.

Throughout the paper, the constant $C$ is generic and may vary from line to line. Unless otherwise stated, the constant $C$ depends only on $n, k, \alpha, \beta$ and the initial hypersurface $X_0$. 
\section{Preliminaries}
\subsection{The normalized flow}
In this subsection, we introduce the normalized procedures. For convenience, we give a unified notation in the setting of both Theorems \ref{thm1} and \ref{thm2}. Suppose $X: M \times [0,T) \rightarrow \mathbb{R}^{n+1}$ solves (\ref{1}) for $f(r) = r^\beta + g(r)$ where $\beta \geq 1 + k\alpha$ and $g(r)$ is a non-negative smooth function, we define the normalized flow $$\tilde{X}(\cdot,t) = \lambda(t)X(\cdot,t)$$ 
where $\lambda$ is defined by
\begin{equation}\label{0}
    \lambda(t) = \begin{cases}
        e^{\gamma t}, \text{ when}\, \, \beta=1+k\alpha, \\
        (1+(\beta-k\alpha-1)\gamma t)^{\frac{1}{\beta - k\alpha - 1}},\text{ when}\, \, \beta>1+k\alpha, 
    \end{cases}
\end{equation}
and $\gamma = (C_n^k)^\alpha$.

By direct calculations,
$$\frac{\p \tilde{X}}{\p t} =   \lambda(t)\frac{\p X}{\p t}+\lambda'(t)X  = -\lambda(t)f(r)\sigma_k^\alpha\nu+\lambda'(t)X$$
Since $\tilde{\sigma}_k^\alpha = \lambda^{-k\alpha }\sigma_k^\alpha,\tilde{r} = \lambda r,\tilde{\nu} = \nu$, it follows that the normalized flow satisfies
\begin{equation*}
    \frac{\p \tilde{X}}{\p t} = -\lambda^{k\alpha + 1}f(\lambda^{-1}\tilde{r})\tilde{\sigma}_k^\alpha \tilde{\nu}+\frac{\lambda'}{\lambda} \tilde{X} = -(\lambda^{k\alpha+1-\beta}\tilde{r}^{\beta} + \lambda^{k\alpha + 1}g(\lambda^{-1}\tilde{r}))\tilde{\sigma}_k^\alpha\tilde{\nu}+\frac{\lambda'}{\lambda} \tilde{X}.
\end{equation*}
Notice that $\frac{\lambda'}{\lambda} = \frac{\gamma }{1+(\beta-k\alpha-1)\gamma t} = \gamma \lambda^{(1+k\alpha)-\beta}$, we have
\begin{equation*}
\begin{split}
    \frac{\p \tilde{X}}{\p t} = &\frac{1}{1+(\beta-k\alpha-1)\gamma t}(-\lambda^{\beta}f(\lambda^{-1}\tilde{r})\tilde{\sigma}_k^\alpha \tilde{\nu}+\gamma \tilde{X})\\
    = &\frac{1}{1+(\beta-k\alpha-1)\gamma t}(-\tilde{r}^{\beta} + \lambda^{\beta}g(\lambda^{-1}\tilde{r}))\tilde{\sigma}_k^\alpha\tilde{\nu}+\gamma \tilde{X}).
    \end{split}
\end{equation*}
Now we define $\tau$ to be 
\begin{equation*}
    \tau = \begin{cases}
        t, \, \, {\rm{when}}\, \beta = k\alpha + 1,\\
        \frac{\log (1 + (\beta - k\alpha - 1)\gamma t)}{(\beta - k\alpha - 1)\gamma}, \, \, {\rm{when}}\,  \beta > k\alpha + 1.
    \end{cases}
\end{equation*}
Then,
\begin{equation*}
\begin{split}
    \frac{\p \tilde{X}}{\p \tau} = &-\lambda^{\beta}f(\lambda^{-1}\tilde{r})\tilde{\sigma}_k^\alpha \tilde{\nu}+\gamma \tilde{X}\\
    = &-(\tilde{r}^{\beta} + \lambda^{\beta}g(\lambda^{-1}\tilde{r}))\tilde{\sigma}_k^\alpha\tilde{\nu}+\gamma \tilde{X}.
    \end{split}
\end{equation*}
For convenience we still use $t$ instead of $\tau$ to denote the time variable and omit the ``tilde" if no confusions arise. The normalized flow is thus written as 
\begin{equation}\label{2}
     \begin{split}
    \frac{\p X}{\p t} = &-\lambda^{\beta}f(\lambda^{-1}r)\sigma_k^\alpha \nu+\gamma X\\
    = &-(r^{\beta} + \lambda^{\beta}g(\lambda^{-1}r))\sigma_k^\alpha\nu+\gamma X.
    \end{split}
\end{equation}

\subsection{Evolution Equations}
In this subsection, we derive some basic evolution equations for the normalized flow \eqref{2}. Denote $$\Phi = \lambda^{\beta}f(\lambda^{-1}r)\sigma_k^\alpha =(r^{\beta} + \lambda^{\beta}g(\lambda^{-1}r))\sigma_k^\alpha.$$
Then the normalized flow can be written as $$\frac{ \p X}{\p t} = -\Phi \nu + \gamma X.$$
We denote the parabolic operator $L = \p_t - \lambda^\beta f(\lambda^{-1}r)\alpha \sigma_k^{\alpha-1}\dot{\sigma}_k^{ij}\n_j\n_i$.
\begin{lemma}[Evolution of the radial distance]
At the point where $r > 0$ we have
    $$\p_t r = -\frac{\Phi u}{r} + \gamma r.$$
\end{lemma}
\begin{proof}
    $$\p_t r = \frac{\l \p_t X, X \r}{r} = -\frac{\Phi u}{r} + \gamma r$$
\end{proof}
\begin{lemma}[Evolution of the metric]\label{lem1}
$$\p_t g_{ij} = -2\Phi h_{ij} + 2\gamma g_{ij}$$
    $$\p_t g^{ij} = 2\Phi h^{ij} - 2\gamma g^{ij}$$
\end{lemma}
\begin{proof}
    Recall $g_{ij} = \l\frac{\p X}{\p x^i},\frac{\p X}{\p x^j}\r$, then
    \begin{equation*}
        \begin{split}
            \p_t g_{ij} = & \l\frac{\p}{\p x^i} \frac{\p X}{\p t},\frac{\p X}{\p x^j}\r + \l\frac{\p X}{\p x^i} ,\frac{\p }{\p x^j}\frac{\p X}{\p t}\r\\
            =&\l\frac{\p}{\p x^i} (-\Phi \nu + \gamma X),\frac{\p X}{\p x^j}\r + \l\frac{\p X}{\p x^i} ,\frac{\p }{\p x^j}(-\Phi \nu + \gamma X)\r\\
            =&-2\Phi h_{ij} + 2\gamma g_{ij}.
        \end{split}
    \end{equation*} 
    Taking derivative over $g_{ik}g^{kj} = \delta_{ij}$, we conclude
    $$\p_tg^{ij} = -g^{ik}g^{jl}\p_tg_{kl}$$
    Hence, 
    $$\p_t g^{ij} = 2\Phi h^{ij} - 2\gamma g^{ij}.$$
\end{proof}
\begin{lemma}[Evolution of the unit normal]
    $$\p_t \nu = \n \Phi$$
\end{lemma}
\begin{proof}
    Since $|\nu| \equiv 1$, it follows $\p_t\nu \perp \nu$. Hence $\p_t\nu$ is tangent to $M_t$.

    Taking derivative with respect to $t$ over $\l \nu,\frac{\p X}{\p x^i} \r \equiv 0$, we have
    $$\l \p_t\nu,\frac{\p X}{\p x^i} \r =- \l \nu,\frac{\p }{\p x^i} \frac{\p X}{\p t} \r = - \l \nu,\frac{\p }{\p x^i} (-\Phi \nu + \gamma X) \r = \frac{\p \Phi}{\p x^i}.$$
    This implies $\p_t\nu = \n \Phi$.
\end{proof}
\begin{lemma}[Evolution of the second fundamental form]\label{lem3}
    \begin{equation}\label{5}
        \p_t h_{ij} = \n_j\n_i \Phi - \Phi h_{i}{ }^kh_{kj} + \gamma h_{ij}
    \end{equation}
    \begin{equation}\label{6}
        \p_t h_i{ }^j = \n^j\n_i \Phi + \Phi h_i{ }^lh_l{ }^j-\gamma h_i{ }^j
    \end{equation}
\end{lemma}
\begin{proof}
    Recall $h_{ij} = -\l \frac{\p^2 X}{\p x^i\p x^j},\nu \r$. Choose normal coordinates at $x \in M$ with respect to the metric $g(t)$.

   Taking the time derivative of $h_{ij}$, we obtain 
    \begin{equation*}
        \begin{split}
            \p_th_{ij} =& -\p_t\l \frac{\p^2 X}{\p x^i\p x^j},\nu \r
            =-\l\frac{\p^2 }{\p x^i\p x^j}\frac{\p X}{\p t},\nu \r - \l \frac{\p^2 X}{\p x^i\p x^j},\p_t\nu \r\\
            =&-\l\frac{\p^2 }{\p x^i\p x^j} (-\Phi \nu + \gamma X),\nu \r \\
            =&\n_i\n_j \Phi + \Phi\l \frac{\p^2 \nu}{\p x^i \p x^j},\nu \r + \gamma h_{ij}\\
            =& \n_i\n_j \Phi - \Phi h_{i}{ }^kh_{kj} + \gamma h_{ij}
        \end{split}
    \end{equation*}
    where we used the following facts at the center of normal coordinates:
    $$\frac{\p^2 \Phi }{\p x^i\p x^j} = \n_i\n_j \Phi,\frac{\p^2 \nu}{\p x^i\p x^j} = \n_ih_j{ }^k\frac{\p X}{\p x^k} - h_j{ }^kh_{ik}\nu,\frac{\p^2 X }{\p x^i\p x^j} = -h_{ij}\nu.$$
    Using Lemma \ref{lem1}, we have
    \begin{equation*}
        \begin{split}
            \p_t h_i{ }^j = &\p_t(h_{ik}g^{kj}) = \p_th_{ik}g^{kj}+h_{ik}\p_tg^{kj}\\
            =&\n^j\n_i \Phi - \Phi h_{i}{ }^kh_{k}{ }^j + \gamma h_{i}{ }^{j} + h_{ik}(2\Phi h^{kj} - 2\gamma g^{kj})\\
            =&\n^j\n_i \Phi + \Phi h_{i}{ }^kh_{k}{ }^j - \gamma h_{i}{ }^{j}
        \end{split}
    \end{equation*}
\end{proof}
\begin{corollary}[Evolution of $\sigma_k^\alpha$]
    $$\p_t(\sigma_k^\alpha) = \alpha \sigma^{\alpha-1}_k \dot{\sigma}_k^{ij} \n_j\n_i \Phi + \alpha \sigma_k^{\alpha-1}\Phi \dot{\sigma}_k^{ij}h_{i}{ }^kh_{kj} - k\alpha \gamma \sigma_k^\alpha$$
\end{corollary}
\begin{proof}
By direct calculations, 
    \begin{equation*}
        \begin{split}
            \p_t(\sigma_k^\alpha) = & \alpha \sigma_k^{\alpha-1}\frac{\p \sigma_k}{\p h_i{ }^j}\p_th_i{ }^j\\
            =&\alpha \sigma_k^{\alpha-1}\frac{\p \sigma_k}{\p h_i{ }^j}(\n^j\n_i \Phi + \Phi h_{i}{ }^kh_{k}{ }^j - \gamma h_{i}{ }^{j})\\
            =&\alpha \sigma^{\alpha-1}_k \dot{\sigma}_k^{ij} \n_j\n_i \Phi + \alpha \sigma_k^{\alpha-1}\Phi \dot{\sigma}_k^{ij}h_{i}{ }^kh_{kj} - k\alpha \gamma \sigma_k^\alpha
        \end{split}
    \end{equation*}
    where we used the fact $\frac{\p \sigma_k}{\p h_i{ }^j}h_{i}{ }^{j} = k \sigma_k$.
\end{proof}

\begin{corollary}[Evolution of $\Phi$]\label{cor1}
\begin{equation*}
    \begin{split}
        L\Phi  =& \lambda^\beta f(\lambda^{-1}r)\alpha \sigma_k^{\alpha-1}\Phi \dot{\sigma}_k^{ij}h_{i}{ }^kh_{kj}
         + (\beta-k\alpha) \gamma \Phi- \lambda^{\beta-1}\frac{\Phi u}{r}f'(\lambda^{-1}r)\sigma_k^\alpha\\
        =& \lambda^\beta f(\lambda^{-1}r)\alpha \sigma_k^{\alpha-1}\Phi \dot{\sigma}_k^{ij}h_{i}{ }^kh_{kj}
         + (\beta-k\alpha) \gamma \Phi-\frac{\Phi^2 u}{r}\frac{f'(\lambda^{-1}r)}{\lambda f(\lambda^{-1}r)}\\
    \end{split}
\end{equation*}
\end{corollary}
\begin{proof}
Recall 
$$\frac{d \lambda}{d\tau} = \frac{d \lambda}{dt}\frac{dt}{d\tau} = \frac{\gamma}{1+(\beta-k\alpha-1)\gamma t}\lambda  (1+(\beta-k\alpha-1)\gamma t) = \gamma \lambda.$$
Therefore,
\begin{equation*}
    \begin{split}
        \p_t\Phi =& \beta \lambda^{\beta-1} \lambda' f(\lambda^{-1}r)\sigma_k^\alpha + \lambda^\beta f'(\lambda^{-1}r)(-\frac{\lambda'}{\lambda^2}r + \lambda^{-1}\p_t r)\sigma_k^\alpha\\
        &+\lambda^\beta f(\lambda^{-1}r)\p_t(\sigma_k^\alpha)\\
        =&\gamma \beta \lambda^\beta f(\lambda^{-1}r)\sigma_k^\alpha + \lambda^\beta f'(\lambda^{-1}r)[-\frac{\gamma}{\lambda}r + \lambda^{-1}(-\frac{\Phi u}{r} + \gamma r)]\sigma_k^\alpha\\
         &+\lambda^\beta f(\lambda^{-1}r)\p_t(\sigma_k^\alpha)\\
         =&\gamma \beta \lambda^\beta f(\lambda^{-1}r)\sigma_k^\alpha - \lambda^{\beta-1}\frac{\Phi u}{r}f'(\lambda^{-1}r)\sigma_k^\alpha\\
         &+\lambda^\beta f(\lambda^{-1}r)(\alpha \sigma^{\alpha-1}_k \dot{\sigma}_k^{ij} \n_j\n_i \Phi + \alpha \sigma_k^{\alpha-1}\Phi \dot{\sigma}_k^{ij}h_{i}{ }^kh_{kj} - k\alpha \gamma \sigma_k^\alpha)\\
         =&\lambda^\beta f(\lambda^{-1}r)\alpha \sigma^{\alpha-1}_k \dot{\sigma}_k^{ij} \n_j\n_i \Phi + \lambda^\beta f(\lambda^{-1}r)\alpha \sigma_k^{\alpha-1}\Phi \dot{\sigma}_k^{ij}h_{i}{ }^kh_{kj}\\
         &+ (\beta-k\alpha) \gamma \Phi- \lambda^{\beta-1}\frac{\Phi u}{r}f'(\lambda^{-1}r)\sigma_k^\alpha
    \end{split}
\end{equation*}
where we used the evolution equations of $r$ and $\sigma_k^\alpha$.
\end{proof}
\begin{lemma}[Evolution of the support function]
   \begin{equation}\label{19}
       \begin{split}
          L u = &  -(1+k\alpha)\Phi + \gamma u +\lambda^{\beta-1}f'(\lambda^{-1}r)\sigma_k^\alpha\l X, \n r \r\\
          & +\lambda^\beta f(\lambda^{-1}r)\alpha\sigma_k^{\alpha-1}\dot{\sigma}_k^{ij}h_i{ }^lh_{lj}u
       \end{split}
   \end{equation} 
\end{lemma}
\begin{proof}
By direct calculations,
    \begin{equation}\label{9}
        \p_tu = \p_t\l X,\nu \r = \l \p_tX,\nu\r + \l X,\p_t \nu \r = -\Phi + \gamma u + \l X,\n \Phi \r
    \end{equation}
Recall
   $$\n_j\n_i u = h_{ij} + \l X,\n_lh_{ij} \frac{\p X}{\p x^l}\r -h_i{ }^lh_{lj}u.$$
    Hence,
    \begin{equation*}
        \begin{split}
           &\lambda^\beta f(\lambda^{-1}r)\alpha\sigma_k^{\alpha-1}\dot{\sigma}_k^{ij}\n_j\n_i u\\
           =&   \lambda^\beta f(\lambda^{-1}r)\alpha\sigma_k^{\alpha-1}\dot{\sigma}_k^{ij}h_{ij}
            +  \lambda^\beta f(\lambda^{-1}r)\alpha\sigma_k^{\alpha-1}\dot{\sigma}_k^{ij}\l X,\n_lh_{ij} \frac{\p X}{\p x^l}\r\\
            &-\lambda^\beta f(\lambda^{-1}r)\alpha\sigma_k^{\alpha-1}\dot{\sigma}_k^{ij}h_i{ }^lh_{lj}u\\
            =&k\alpha\Phi
            +  \lambda^\beta f(\lambda^{-1}r)\l X,\n \sigma_k^\alpha\r
            -\lambda^\beta f(\lambda^{-1}r)\alpha\sigma_k^{\alpha-1}\dot{\sigma}_k^{ij}h_i{ }^lh_{lj}u\\
        \end{split}
    \end{equation*}
Recall $\n \Phi = \lambda^{\beta-1}f'(\lambda^{-1}r)\n r \sigma_k^\alpha + \lambda^\beta f(\lambda^{-1}r)\n \sigma_k^\alpha$, it follows
\begin{equation*}
    \begin{split}
         \lambda^\beta f(\lambda^{-1}r)\alpha\sigma_k^{\alpha-1}\dot{\sigma}_k^{ij}\n_j\n_i u
         =&k\alpha\Phi
            +  \l X,\n \Phi \r - \lambda^{\beta-1}f'(\lambda^{-1}r)\sigma_k^\alpha\l X, \n r \r \\
            &-\lambda^\beta f(\lambda^{-1}r)\alpha\sigma_k^{\alpha-1}\dot{\sigma}_k^{ij}h_i{ }^lh_{lj}u\\
    \end{split}
\end{equation*}
Combining the above equation with \eqref{9}, we derive the desired equation.
\end{proof}
At the end of this section, we derive another useful form of the evolution equation of the second fundamental form.
\begin{lemma}[Evolution of the second fundamental form]
  \begin{equation*}
        \begin{split}
            L h_{ij} = &   \lambda^{\beta-2} f''(\lambda^{-1}r)\n_jr\n_ir \sigma_k^\alpha \\
        &+ \lambda^{\beta-1} f'(\lambda^{-1}r)\n_ir \n_j\sigma_k^\alpha+\lambda^{\beta-1} f'(\lambda^{-1}r)\n_jr \n_i\sigma_k^\alpha \\
        &+  \lambda^\beta f(\lambda^{-1}r) \alpha(\alpha - 1)\sigma_k^{\alpha - 2} \dot{\sigma}_k^{st}\n_jh_{st}\dot{\sigma}_k^{pq}\n_ih_{pq} \\
        &+\lambda^\beta f(\lambda^{-1}r) \alpha \sigma_k^{\alpha - 1}\ddot{\sigma}_k^{pq,st}\n_ih_{pq}\n_jh_{st} \\
    & +  \lambda^\beta f(\lambda^{-1}r) \alpha \sigma_k^{\alpha-1}\dot{\sigma}_k^{pq}h_i{ }^sh_{sp}h_{jq}+\lambda^{\beta-1}f'(\lambda^{-1}r)\n_j\n_ir \sigma_k^\alpha\\
    &-  \lambda^\beta f(\lambda^{-1}r) k\alpha \sigma_k^{\alpha}\dot{\sigma}_k^{pq}h_i{ }^sh_{sj} +  \lambda^\beta f(\lambda^{-1}r) \alpha \sigma_k^{\alpha-1}\dot{\sigma}_k^{pq}h_{p}{ }^sh_{sq}h_{ij}\\
    &-  \lambda^\beta f(\lambda^{-1}r) \alpha \sigma_k^{\alpha-1}\dot{\sigma}_k^{pq}h_p{ }^sh_{sj}h_{qi}- \Phi h_{i}{ }^kh_{kj} + \gamma h_{ij}
        \end{split}
    \end{equation*}
         \begin{equation}\label{18}
         \begin{split}
            Lh_i{ }^j = & \lambda^\beta f(\lambda^{-1}r) \alpha(\alpha - 1)\sigma_k^{\alpha - 2} \n^j\sigma_k\n_i\sigma_k \\
        &+\lambda^\beta f(\lambda^{-1}r) \alpha \sigma_k^{\alpha - 1}\ddot{\sigma}_k^{pq,st}\n_ih_{pq}\n^jh_{st}-\gamma h_i{ }^j \\
         &- (k\alpha-1) \Phi h_i{ }^sh_{s}{ }^{j} +  \lambda^\beta f(\lambda^{-1}r) \alpha \sigma_k^{\alpha-1}\dot{\sigma}_k^{pq}h_{p}{ }^sh_{sq}h_{i}{ }^{j}\\
         &+\lambda^{\beta-1} f'(\lambda^{-1}r)\n_ir \n^j\sigma_k^\alpha+\lambda^{\beta-1} f'(\lambda^{-1}r)\n^jr \n_i\sigma_k^\alpha \\
        & +\lambda^{\beta-1} f'(\lambda^{-1}r)\n^j\n_ir \sigma_k^\alpha+ \lambda^{\beta-2} f''(\lambda^{-1}r)\n^jr\n_ir \sigma_k^\alpha
         \end{split}
     \end{equation}
\end{lemma}
\begin{proof}
Recall $\Phi = \lambda^\beta f(\lambda^{-1}r)\sigma_k^\alpha$, we have
$$\n_i\Phi=\lambda^{\beta-1} f'(\lambda^{-1}r)\n_ir \sigma_k^\alpha + \lambda^\beta f(\lambda^{-1}r) \n_i\sigma_k^\alpha,$$
\begin{equation*}\label{14}
    \begin{split}
        \n_j\n_i\Phi=&\lambda^{\beta-2} f''(\lambda^{-1}r)\n_jr\n_ir \sigma_k^\alpha+\lambda^{\beta-1} f'(\lambda^{-1}r)\n_j\n_ir \sigma_k^\alpha \\
        &+ \lambda^{\beta-1} f'(\lambda^{-1}r)\n_ir \n_j\sigma_k^\alpha\\
        &+\lambda^{\beta-1} f'(\lambda^{-1}r)\n_jr \n_i\sigma_k^\alpha + \lambda^\beta f(\lambda^{-1}r) \n_j\n_i\sigma_k^\alpha.
    \end{split}
\end{equation*}
  By a direct calculation,
  \begin{equation*}
  \begin{split}
    \n_j\n_i \sigma_k^\alpha = & \alpha(\alpha - 1)\sigma_k^{\alpha - 2} \dot{\sigma}_k^{st}\n_jh_{st}\dot{\sigma}_k^{pq}\n_ih_{pq} + \alpha \sigma_k^{\alpha - 1}\ddot{\sigma}_k^{pq,st}\n_ih_{pq}\n_jh_{st}\\
    &+ \alpha \sigma_k^{\alpha-1}\dot{\sigma}_k^{pq}\n_j\n_ih_{pq}.
  \end{split}
  \end{equation*}
Recall the Ricci identity,
    $$\n_j\n_ih_{pq} = \n_q\n_ph_{ij} + h_i{ }^sh_{sp}h_{jq} - h_i{ }^sh_{sj}h_{pq} + h_{ij}h_{q}{ }^sh_{sp} - h_q{ }^sh_{sj}h_{pi}.$$
    Hence,
\begin{equation*}
\begin{split}
    \n_j\n_i \sigma_k^\alpha = & \alpha(\alpha - 1)\sigma_k^{\alpha - 2} \dot{\sigma}_k^{st}\n_jh_{st}\dot{\sigma}_k^{pq}\n_ih_{pq} 
    + \alpha \sigma_k^{\alpha - 1}\ddot{\sigma}_k^{pq,st}\n_ih_{pq}\n_jh_{st} \\
    &+ \alpha \sigma_k^{\alpha-1}\dot{\sigma}_k^{pq}\n_q\n_ph_{ij} + \alpha \sigma_k^{\alpha-1}\dot{\sigma}_k^{pq}h_i{ }^sh_{sp}h_{jq}\\
    &- k\alpha \sigma_k^{\alpha}\dot{\sigma}_k^{pq}h_i{ }^sh_{sj} + \alpha \sigma_k^{\alpha-1}\dot{\sigma}_k^{pq}h_{p}{ }^sh_{sq}h_{ij} - \alpha \sigma_k^{\alpha-1}\dot{\sigma}_k^{pq}h_p{ }^sh_{sj}h_{qi}.
\end{split}
\end{equation*}
Now putting this into \eqref{14}, we have
\begin{equation*}
    \begin{split}
        \n_j\n_i\Phi=&\lambda^{\beta-2} f''(\lambda^{-1}r)\n_jr\n_ir \sigma_k^\alpha+\lambda^{\beta-1} f'(\lambda^{-1}r)\n_j\n_ir \sigma_k^\alpha \\
        &+ \lambda^{\beta-1} f'(\lambda^{-1}r)\n_ir \n_j\sigma_k^\alpha+\lambda^{\beta-1} f'(\lambda^{-1}r)\n_jr \n_i\sigma_k^\alpha \\
        &+  \lambda^\beta f(\lambda^{-1}r) \alpha(\alpha - 1)\sigma_k^{\alpha - 2} \dot{\sigma}_k^{st}\n_jh_{st}\dot{\sigma}_k^{pq}\n_ih_{pq} \\
        &+\lambda^\beta f(\lambda^{-1}r) \alpha \sigma_k^{\alpha - 1}\ddot{\sigma}_k^{pq,st}\n_ih_{pq}\n_jh_{st} \\
    &+  \lambda^\beta f(\lambda^{-1}r) \alpha \sigma_k^{\alpha-1}\dot{\sigma}_k^{pq}\n_q\n_ph_{ij} +  \lambda^\beta f(\lambda^{-1}r) \alpha \sigma_k^{\alpha-1}\dot{\sigma}_k^{pq}h_i{ }^sh_{sp}h_{jq}\\
    &-  \lambda^\beta f(\lambda^{-1}r) k\alpha \sigma_k^{\alpha}\dot{\sigma}_k^{pq}h_i{ }^sh_{sj} +  \lambda^\beta f(\lambda^{-1}r) \alpha \sigma_k^{\alpha-1}\dot{\sigma}_k^{pq}h_{p}{ }^sh_{sq}h_{ij}\\
    &-  \lambda^\beta f(\lambda^{-1}r) \alpha \sigma_k^{\alpha-1}\dot{\sigma}_k^{pq}h_p{ }^sh_{sj}h_{qi}.
    \end{split}
\end{equation*}
Substituting the above equation into Lemma \ref{lem3} we get the desired result.
   
\end{proof}
As a corollary, we derive the evolution equations of $|A|^2$ and $H$ which will be used in $C^2$ estimate.
\begin{lemma}[Evolution equation of the $|A|^2$]
     \begin{equation}\label{20}
        \begin{split}
            L|A|^2=& 2\lambda^\beta f(\lambda^{-1}r) \alpha(\alpha - 1)\sigma_k^{\alpha - 2} \n^j\sigma_k\n_i\sigma_k h_j{ }^i \\
        &+2\lambda^\beta f(\lambda^{-1}r) \alpha \sigma_k^{\alpha - 1}\ddot{\sigma}_k^{pq,st}\n_ih_{pq}\n^jh_{st}h_j{ }^i-2\gamma |A|^2 \\
         &- 2(k\alpha-1) \Phi h_i{ }^sh_{s}{ }^{j}h_j{ }^i +  2\lambda^\beta f(\lambda^{-1}r) \alpha \sigma_k^{\alpha-1}\dot{\sigma}_k^{pq}h_{p}{ }^sh_{sq}|A|^2\\
         &+ 2\lambda^{\beta-1} f'(\lambda^{-1}r)\n_ir \n^j\sigma_k^\alpha h_j{ }^i+2\lambda^{\beta-1} f'(\lambda^{-1}r)\n^jr \n_i\sigma_k^\alpha h_j{ }^i \\
        & +2\lambda^{\beta-1} f'(\lambda^{-1}r)\n^j\n_ir h_j{ }^i \sigma_k^\alpha+ 2\lambda^{\beta-2} f''(\lambda^{-1}r)\n^jr\n_ir \sigma_k^\alpha h_j{ }^i\\
         &-2\lambda^\beta f(\lambda^{-1}r) \alpha \sigma_k^{\alpha-1}\dot{\sigma}_k^{pq} \n_p h_i{ }^j\n_qh_j{ }^i
        \end{split}
    \end{equation}
\end{lemma}
\begin{proof}
    Recall $|A|^2 = h_i{ }^j h_j{ }^i$, hence 
    \begin{equation*}
        \begin{split}
            &(\p_t - \lambda^\beta f(\lambda^{-1}r) \alpha \sigma_k^{\alpha-1}\dot{\sigma}_k^{pq}\n_q\n_p)|A|^2\\
            =&h_{i}{ }^j (\p_t - \lambda^\beta f(\lambda^{-1}r) \alpha \sigma_k^{\alpha-1}\dot{\sigma}_k^{pq}\n_q\n_p)h_j{ }^i\\
            &+  (\p_t - \lambda^\beta f(\lambda^{-1}r) \alpha \sigma_k^{\alpha-1}\dot{\sigma}_k^{pq}\n_q\n_p)h_i{ }^j h_j{ }^i\\
            &-2\lambda^\beta f(\lambda^{-1}r) \alpha \sigma_k^{\alpha-1}\dot{\sigma}_k^{pq} \n_p h_i{ }^j\n_qh_j{ }^i\\
            =& 2\lambda^\beta f(\lambda^{-1}r) \alpha(\alpha - 1)\sigma_k^{\alpha - 2} \n^j\sigma_k\n_i\sigma_k h_j{ }^i \\
        &+2\lambda^\beta f(\lambda^{-1}r) \alpha \sigma_k^{\alpha - 1}\ddot{\sigma}_k^{pq,st}\n_ih_{pq}\n^jh_{st}h_j{ }^i-2\gamma |A|^2 \\
         &- 2(k\alpha-1) \Phi h_i{ }^sh_{s}{ }^{j}h_j{ }^i +  2\lambda^\beta f(\lambda^{-1}r) \alpha \sigma_k^{\alpha-1}\dot{\sigma}_k^{pq}h_{p}{ }^sh_{sq}|A|^2\\
         &+ 2\lambda^{\beta-1} f'(\lambda^{-1}r)\n_ir \n^j\sigma_k^\alpha h_j{ }^i+2\lambda^{\beta-1} f'(\lambda^{-1}r)\n^jr \n_i\sigma_k^\alpha h_j{ }^i \\
        & +2\lambda^{\beta-1} f'(\lambda^{-1}r)\n^j\n_ir h_j{ }^i \sigma_k^\alpha+ 2\lambda^{\beta-2} f''(\lambda^{-1}r)\n^jr\n_ir \sigma_k^\alpha h_j{ }^i\\
         &-2\lambda^\beta f(\lambda^{-1}r) \alpha \sigma_k^{\alpha-1}\dot{\sigma}_k^{pq} \n_p h_i{ }^j\n_qh_j{ }^i
        \end{split}
    \end{equation*}
\end{proof}
\begin{lemma}[Evolution of $H$]
\begin{equation*}
    \begin{split}
    \p_t H = & \lambda^\beta f(\lambda^{-1}r) \alpha \sigma_k^{\alpha-1}\dot{\sigma}_k^{pq}\n_q\n_pH +  \lambda^\beta f(\lambda^{-1}r) \alpha(\alpha - 1)\sigma_k^{\alpha - 2} |\n \sigma_k|^2 \\
        &+\lambda^\beta f(\lambda^{-1}r) \alpha \sigma_k^{\alpha - 1}\ddot{\sigma}_k^{pq,st}\n_ih_{pq}\n^ih_{st}-\gamma H - (k\alpha-1) \Phi |A|^2\\
        &+  \lambda^\beta f(\lambda^{-1}r) \alpha \sigma_k^{\alpha-1}\dot{\sigma}_k^{pq}h_{p}{ }^sh_{sq}H+ 2\lambda^{\beta-1} f'(\lambda^{-1}r)\l \n r, \n \sigma_k^\alpha \r \\
        & +\lambda^{\beta-1} f'(\lambda^{-1}r)\Delta r \sigma_k^\alpha+ \lambda^{\beta-2} f''(\lambda^{-1}r)|\n r|^2\sigma_k^\alpha
         \end{split}
     \end{equation*}
\end{lemma}
\begin{proof}
    Taking trace over \eqref{18}.
\end{proof}
\subsection{Equations as radial graph}
We first recall some basic formulae for star-shaped hypersurfaces in $\R^{n+1}$.

Let $M = \{r(\theta)\theta: \theta \in \s^n\}$ where $r: \s^n \rightarrow \R_+$ is a smooth positive function. Then the geometric quantities of $M$ can be represented by
\begin{equation}\label{4}
    \begin{aligned}
u & =\frac{r}{\rho}, \quad g_{i j}=r^2\left(g^{\s^n}_{i j}+\phi_i \phi_j\right), \quad g^{i j}=\frac{1}{r^2}\left(g^{ij}_{\s^n}-\frac{\phi^i \phi^j}{\rho^2}\right), \\
h_{i j} & =\frac{r}{\rho}\left(g^{\s^n}_{i j}+\phi_i \phi_j-\phi_{i j}\right), \quad h_i^j=\frac{1}{r \rho}\left(\delta_{i}{ } ^{j}-\phi_{i}{ }^{j}+\frac{\phi_{i l} \phi^l \phi^j}{\rho^2}\right)
\end{aligned}
\end{equation}
where $\rho = \sqrt{1 + |\n \phi|^2}$ and $\phi = \log r$. The normalized flow \eqref{2} can be reduced to the following scalar equations on $\s^n$(see \cite{li2020asymptotic})
\begin{equation}\label{19} 
    \p_t r = -(r^{\beta} + \lambda^{\beta}g(\lambda^{-1}r))\frac{\sqrt{r^2 + |\nabla r|^2}}{r}\sigma_k^\alpha\left(\frac{1}{r \rho}(\delta_{i}{ } ^{j}-\phi_{i}{ }^{j}+\frac{\phi_{i l} \phi^l \phi^j}{\rho^2})\right) + \gamma r,
\end{equation}
\begin{equation}\label{3}
    \p_t \phi = -(e^{(\beta-1)\phi}\rho + \frac{\rho}{r}\lambda^{\beta}g(\lambda^{-1}r))\sigma_k^\alpha\left(\frac{1}{r \rho}(\delta_{i}{ } ^{j}-\phi_{i}{ }^{j}+\frac{\phi_{i l} \phi^l \phi^j}{\rho^2})\right) + \gamma.
\end{equation}

\section{$C^0$-estimate}
In this section, we give the $C^0$ estimate for the normalized flow.
\begin{lemma}\label{lem5}
    Along the normalized flow (\ref{19}), the radial distance $r$ has uniform upper and lower bounds, i.e., $\exists C > 0$ such that
    $$0 < \frac{1}{C} \leq r \leq C < \infty$$
\end{lemma}
\begin{proof}
\begin{itemize}
 \item {\it Upper Bound}
    
     At the maximum point of $r$ we have $\nabla r = 0$ and $\n^2r \leq 0$.
    By \eqref{4}, we conclude $g_{ij} = r^2g^{\s^n}_{ij}, h_{ij} = r( g_{ij}^{\s^n}-\phi_{ij}) \geq  r g_{ij}^{\s^n}$. Therefore,
    $$\sigma_k^\alpha \geq \frac{\gamma}{r^{k\alpha}} > 0.$$
By (\ref{19})
\begin{equation*}
\begin{split}
    \frac{d}{dt}r_{\max} &= -r_{\max}^{\beta}\sigma_k^\alpha -  \lambda^\beta g(\lambda^{-1}r_{\max})\sigma_k^\alpha + \gamma r_{\max}\\
    &\leq -\gamma r^{\beta - k\alpha}_{\max} + \gamma r_{\max}
\end{split}
\end{equation*}
where we used $g \geq 0$. This implies $r_{\max}(t) \leq C$ since $\beta - k\alpha \geq 1$.

    \item {\it Lower Bound}
    
    At the minimum point of $r$ we have $\nabla r = 0$ and $\n^2r \geq 0$. Similarly,
    $$\sigma_k^\alpha \leq \frac{\gamma}{r^{k\alpha}}.$$

Hence, by \eqref{19}
\begin{equation*}
\begin{split}
     \frac{d}{dt} r_{\min} &= -(r_{\min}^{\beta} + \lambda^\beta g(\lambda^{-1}r_{\min}))\sigma_k^\alpha + \gamma r_{\min}\\
    &\geq -\gamma r_{\min}^{\beta - k\alpha}-  \lambda^\beta g(\lambda^{-1}r_{\min}))\sigma_k^\alpha+ \gamma r_{\min}.
\end{split}
\end{equation*}

\textbf{Case 1. $\beta = 1+k\alpha$.}  Then,
\begin{equation*}
     \frac{d}{dt} r_{\min}\geq -  \lambda^\beta g(\lambda^{-1}r_{\min})\sigma_k^\alpha.
\end{equation*}
Recall we assume $g \equiv 0$ on $[0,\epsilon]$. Now we claim $$r_{\min}(t) \geq \min\{r_{\min}(0),\epsilon\}.$$
Suppose not, then there exists a $t_0$ such that 
$$r_{\min}(t_0) = \mu \in (0,\min\{r_{\min}(0),\epsilon\}).$$
Since $r_{\min}(0) \geq \min\{r_{\min}(0),\epsilon\}$, we can assume there exists $t_1 < t_0$ such that $r_{\min}(t_1) = \min\{r_{\min}(0),\epsilon\}$ and $\mu \leq r_{\min}(t) \leq \min\{r_{\min}(0),\epsilon\}$ on $[t_1,t_0]$.
Notice that $\lambda^{-1} \leq 1$, we have $\lambda^{-1}r_{\min}(t) \leq \epsilon$ for $t \in [t_1,t_0]$. Therefore, $\frac{d}{dt} r_{\min} \geq 0,\forall t \in [t_1,t_0]$ and hence 
$$\mu > r_{\min}(t_0) \geq r_{\min}(t_1) \geq \min\{r_{\min}(0),\epsilon\}.$$
Contradiction!

\textbf{Case 2. $\beta > k\alpha + 1$.} By upper bound of $r$, we have $r_{\min}(t) \leq R$ for some $0<R< \infty$. Since we assume $g'(0) = \cdots = g^{([\beta])}(0) = 0$, it follows $g(r) \leq C r^{[\beta]+1},\forall r \in [0,R]$ for some constant $C$ and hence
$$g(\lambda^{-1}r_{\min}) \leq C \lambda^{-([\beta]+1)}r_{\min}^{[\beta]+1}.$$
Therefore,
\begin{equation*}
\begin{split}
     \frac{d}{dt} r_{\min} &\geq -\gamma r_{\min}^{\beta - k\alpha}-  C\gamma\lambda^{\beta-([\beta]+1)} r_{\min}^{[\beta]+1-k\alpha}+ \gamma r_{\min}\\
     \geq&-\gamma r_{\min}^{\beta - k\alpha}-  C\gamma r_{\min}^{[\beta]+1-k\alpha}+ \gamma r_{\min}.
\end{split}
\end{equation*}
This implies $r_{\min}(t) \geq C$ since $[\beta]+1 -k\alpha > \beta  - k\alpha > 1$.
\end{itemize}
\end{proof}
At the end of this section, we give two consequences due to the $C^0$ estimate. One will be used to get the $C^2$ estimate, and the another will be used to show the flow \eqref{2} preserves the $k$-convexity.

\begin{corollary}\label{cor 2}
Along the normalized flow \eqref{2}, we have the following estimates:
    \begin{equation}\label{10}
        0 < \frac{1}{C} \leq \lambda^\beta f(\lambda^{-1}r) \leq C,
    \end{equation}
    \begin{equation}\label{11}
        \lambda^{\beta-1} f'(\lambda^{-1}r) \leq C,
    \end{equation}
    \begin{equation}\label{12}
      \frac{f'(\lambda^{-1}r)}{\lambda f(\lambda^{-1}r)} \leq C,
    \end{equation}
   \begin{equation}\label{15}
       \lambda^{\beta-2} |f''(\lambda^{-1}r)| \leq C.
   \end{equation}
\end{corollary}
\begin{proof}

   {Proof of \eqref{10}}.  Since $r$ has a positive lower bound by $C^0$ estimate, it follows 
    $$\lambda^\beta f(\lambda^{-1}r) = r^{\beta} + \lambda^\beta g(\lambda^{-1}r) \geq r^{\beta} \geq \frac{1}{C} > 0.$$
    On the other hand, since $r$ has an upper bound $R < \infty$, by Taylor's formula we have
    $$g(r) \leq C r^{[\beta]+1},\forall r \in [0,R].$$
    Hence,
    $$\lambda^\beta g(\lambda^{-1}r)\leq C\lambda^{\beta-([\beta]+1)} r^{[\beta]+1} \leq C R^{[\beta]+1}.$$
    
    { Proof of \eqref{11}}. Recall $f'(r) = \beta r^{\beta-1} + g'(r)$,hence
    \begin{equation*}
        \begin{split}
           \lambda^{\beta-1}|f'(\lambda^{-1}r)| = & |\beta r^{\beta-1} + \lambda^{\beta-1}g'(\lambda^{-1}r)|\\
            \leq & \beta r^{\beta-1} + \lambda^{\beta-1}|g'(\lambda^{-1}r)|
            \end{split}
    \end{equation*}
    Since $r$ has an upper bound $R < \infty$, by Taylor's formula
    $$g'(r) \leq C r^{[\beta]},\forall r \in [0,R]$$
    Therefore,
    $$\lambda^{\beta-1}|g'(\lambda^{-1}r)| \leq C \lambda^{\beta-1 - [\beta]}r^{[\beta]} \leq C R^{[\beta]}$$
    
    {Proof of \eqref{12}}. By a direct calculation,
    \begin{equation*}
        \begin{split}
           \frac{|f'(\lambda^{-1}r)|}{\lambda f(\lambda^{-1}r)}  = & \lambda^{-1}\frac{|\beta \lambda^{-(\beta-1)} r^{\beta-1} + g'(\lambda^{-1}r)|}{\lambda^{-\beta}r^\beta + g(\lambda^{-1}r)}\\
            = & \lambda^{-1}\frac{|\beta \lambda r^{\beta-1} + \lambda^\beta g'(\lambda^{-1}r)|}{r^{\beta} + \lambda^\beta g(\lambda^{-1}r)}\\
            =& \frac{|\beta  r^{\beta-1} + \lambda^{\beta-1} g'(\lambda^{-1}r)|}{r^{\beta} + \lambda^\beta g(\lambda^{-1}r)}\\
            \leq &  \frac{\beta  r^{\beta-1} + \lambda^{\beta-1} |g'(\lambda^{-1}r)|}{r^{\beta}}.
        \end{split}
    \end{equation*}
    By the proof of the previous one, we have $\lambda^{\beta-1} |g'(\lambda^{-1}r)| \leq C$ and hence we get the desired bound.
    
    {Proof of \eqref{15}}. Recall $f''(r) = \beta (\beta-1)r^{\beta-2} + g''(r)$, hence
    $$ \lambda^{\beta-2} |f''(\lambda^{-1}r)| \leq \beta (\beta -1 )r^{\beta-2} + \lambda^{\beta-2} |g''(\lambda^{-1}r)|. $$
    Since $r$ has an upper bound $R < \infty$, by Taylor's formula we have
    $$g''(r) \leq C r^{[\beta]-1},\forall r \in [0,R].$$
    Therefore,
    $$\lambda^{\beta-2} |g''(\lambda^{-1}r)| \leq C \lambda^{\beta-1 -[\beta]}r^{[\beta]-1} \leq C R^{[\beta]-1}.$$
    \end{proof}
    
\begin{corollary}
    The $k$-convexity is preserved along the normalized flow \eqref{2}. 
    \end{corollary}
\begin{proof}
    Recall from Corollary \ref{cor1}
   \begin{equation*}
    \begin{split}
        L\Phi  = \lambda^\beta f(\lambda^{-1}r)\alpha \sigma_k^{\alpha-1}\Phi \dot{\sigma}_k^{ij}h_{i}{ }^kh_{kj}
         + (\beta-k\alpha) \gamma \Phi-\frac{\Phi^2 u}{r}\frac{f'(\lambda^{-1}r)}{\lambda f(\lambda^{-1}r)}\\
    \end{split}
\end{equation*}
Suppose that there exists a first time $t_0 > 0$ and a point $x_0 \in M$ such that $\Phi(x_0,t_0) = 0$ with $\Phi > 0$ on $[0,t_0)$.

On the time interval $[0,t_0]$, we can assume $|u| \leq C$ where $C$ may depend on $t_0$. In a view of \eqref{12}, we conclude
$$\frac{d}{dt}\Phi_{\min} \geq \gamma \Phi - C \Phi^2$$
which implies $\Phi \geq C > 0$ on $[0,t_0]$. This contradicts the fact that $\Phi(x_0,t_0) = 0$.
Therefore, $\Phi > 0$ along the normalized flow and hence $\sigma_k^\alpha > 0$.
\end{proof}

\section{$C^1$-estimate}

In this section, we give the $C^1$ estimate for the normalized flow. Note that the $\nabla$ in this section denotes the Levi-Civita connection on $\s^n$ with respect to the standard metric. 
\begin{lemma}
    Along the normalized flow (\ref{3}), we have
    $$|\n\phi| \leq C.$$
\end{lemma}
\begin{proof}
    Let $\theta_0$ be a maximum point of $|\n\phi|^2(\cdot,t)$, then at this point by (\ref{3}) we conclude
\begin{equation}\label{7}
    \begin{split}
        \p_t(\frac{1}{2}|\n\phi|^2) = \phi_i(\p_t\phi)_i=&-\phi_i\{ (e^{k\alpha\phi}\rho + \frac{\rho}{r}\lambda^\beta g(\lambda^{-1}r))\sigma_k^\alpha\}_i \\
        =& -\phi_i(e^{(\beta-1)\phi} (\beta-1)\phi_i \sigma_k^\alpha \rho+e^{(\beta-1)\phi}\rho (\sigma_k^\alpha)_i)\\
        &+\phi_i \frac{\rho}{r^2}r_i\lambda^\beta g(\lambda^{-1}r)\sigma_k^\alpha\\
        & -\phi_i \frac{\rho}{r}\lambda^{\beta-1} g'(\lambda^{-1}r)r_i\sigma_k^\alpha- \phi_i\frac{\rho}{r}\lambda^\beta g(\lambda^{-1}r)(\sigma_k^\alpha)_i \\
        =& -(\beta-1) e^{(\beta-1)\phi} |\n\phi|^2 \sigma_k^\alpha \rho-e^{(\beta-1)\phi}\rho\phi_i(\sigma_k^\alpha)_i\\
        &+|\n\phi|^2 \frac{\rho}{r}\lambda^\beta g(\lambda^{-1}r)\sigma_k^\alpha -|\n\phi|^2 \rho \lambda^{\beta-1} g'(\lambda^{-1}r) \sigma_k^\alpha \\
        &- \frac{\rho}{r}\lambda^\beta g(\lambda^{-1}r)\phi_i(\sigma_k^\alpha)_i
    \end{split}
\end{equation}
where we used  $(\sqrt{1 + |\n\phi|^2})_i = 0$ at $(\theta_0,t)$ and $\phi_i = \frac{r_i}{r}$.

By a direct calculation, 
\begin{equation*}
    \begin{split}
        (\sigma_k^\alpha)_i = &\alpha \sigma_k^{\alpha-1} \frac{\p \sigma_k}{\p h_p{ }^q}\n_i h_p{ }^q\\
        =&\alpha \sigma_k^{\alpha-1}\frac{\p \sigma_k}{\p h_p{ }^q}\n_ih_{pm}g^{mq} + \alpha \sigma_k^{\alpha-1}\frac{\p \sigma_k}{\p h_p{ }^q}h_{pm}\n_ig^{mq}\\
        =&\alpha \sigma_k^{\alpha-1}\dot{\sigma}_k^{pq}\n_ih_{pq} - \alpha \sigma_k^{\alpha-1}\frac{\p \sigma_k}{\p h_p{ }^q}h_{pm}g^{ms}g^{qt}\n_ig_{st}\\
        =&\alpha \sigma_k^{\alpha-1}\dot{\sigma}_k^{pq}(\n_ih_{pq}-h_{pm}g^{ms}\n_ig_{sq}).
    \end{split}
\end{equation*}
Use (\ref{4}), we have
$$\n_ig_{sq} = 2\phi_i g_{sq} + e^{2\phi}\phi_{si}\phi_q + e^{2\phi}\phi_s\phi_{qi}$$
$$\n_ih_{pq} = \phi_ih_{pq} + \frac{e^\phi}{\rho}(\phi_{pi}\phi_q + \phi_p\phi_{qi} - \phi_{pqi})$$
where we used the fact that $\rho_i = 0$ at $(\theta_0,t)$.

Therefore,
\begin{equation*}
    \begin{split}
        (\sigma_k^\alpha)_i =& \alpha \sigma_k^{\alpha-1}\dot{\sigma}_k^{pq}\bigg(\phi_ih_{pq} + \frac{e^\phi}{\rho}(\phi_{pi}\phi_q + \phi_p\phi_{qi} - \phi_{pqi})\\
     &-h_{pm}g^{ms}(2\phi_i g_{sq} + e^{2\phi}\phi_{si}\phi_q + e^{2\phi}\phi_s\phi_{qi})\bigg).
    \end{split}
\end{equation*}
At $(\theta_0,t)$, we have $\phi_{pi}\phi_p = 0$. Hence,
\begin{equation*}
        \phi_i(\sigma_k^\alpha)_i = -k\alpha \sigma_k^{\alpha}|\n\phi|^2 -\alpha \sigma_k^{\alpha-1}\dot{\sigma}_k^{pq}\frac{e^\phi}{\rho} \phi_i\phi_{pqi}.
\end{equation*}
where we used $\dot{\sigma_k}^{pq}h_{pq} = k\sigma_k$.

By the Ricci identity, Codazzi equation and Gauss equation,

$$
\begin{aligned}
\phi_{pqi}=&\phi_{piq}+\phi_j R_{jpq i}^{\mathbb{S}^n}\\
 =&\phi_{ipq}+\phi_j\left(\delta_{jq} \delta_{pi}-\delta_{ji} \delta_{pq}\right) \\
=&\phi_{ipq}+\phi_q \delta_{pi}-\phi_i \delta_{pq}.
\end{aligned}
$$
Thus,
$$
\begin{aligned}
    \phi_i \phi_{pqi}=&\phi_i \phi_{ipq}+\phi_p\phi_q-\delta_{pq}|\n \phi|^2\\
    =&\left(\frac{1}{2}|\n \phi|^2\right)_{pq}-\phi_{ip} \phi_{iq}+\phi_p \phi_q-\delta_{pq}|\n \phi|^2.
\end{aligned}
$$

Hence we have

\begin{equation*}
\phi_i(\sigma_k^\alpha)_i=  -k\alpha \sigma_k^{\alpha}|\n\phi|^2  
 - \alpha \sigma_k^{\alpha-1}\dot{\sigma}_k^{pq}\frac{e^\phi}{\rho}\left(\left(\frac{1}{2}|\n \phi|^2\right)_{pq}-\phi_{ip} \phi_{iq}+\phi_p \phi_q-\delta_{pq}|\n \phi|^2 \right).
\end{equation*}
Now putting these into (\ref{7}), we have
\begin{equation*}
    \begin{split}
        \p_t(\frac{1}{2}|\n\phi|^2) 
        \leq& -(\beta - k\alpha - 1)\rho \sigma_k^\alpha |\n \phi|^2 e^{(\beta-1)\phi}\\
        &+\alpha \sigma_k^{\alpha-1}\dot{\sigma}_k^{pq}e^{\beta\phi}\left(\left(\frac{1}{2}|\n \phi|^2\right)_{pq}-\phi_{ip} \phi_{iq}+\phi_p \phi_q-\delta_{pq}|\n \phi|^2 \right)  \\
        &+(1+k\alpha)|\n\phi|^2 \frac{\rho}{r}\lambda^\beta g(\lambda^{-1}r)\sigma_k^\alpha  -|\n\phi|^2 \rho \lambda^{\beta - 1} g'(\lambda^{-1}r)\sigma_k^\alpha\\
        &+ \alpha \sigma_k^{\alpha-1}\dot{\sigma}_k^{pq}\lambda^{\beta}g(\lambda^{-1}r)\left(\left(\frac{1}{2}|\n \phi|^2\right)_{pq}-\phi_{ip} \phi_{iq}+\phi_p \phi_q-\delta_{pq}|\n \phi|^2 \right).
    \end{split}
\end{equation*}
Notice that
\begin{equation*}
    \begin{split}
        &(1+k\alpha)|\n\phi|^2 \frac{\rho}{r}\lambda^\beta g(\lambda^{-1}r)\sigma_k^\alpha  -|\n\phi|^2 \rho \lambda^{\beta - 1} g'(\lambda^{-1}r)\sigma_k^\alpha\\
        &= |\n \phi|^2\sigma_k^\alpha \lambda^{\beta-1}\rho((1+k\alpha)\frac{g(\lambda^{-1}r)}{\lambda^{-1}r}-g'(\lambda^{-1}r)).
    \end{split}
\end{equation*}
By our assumption on $g$, we have 
$$(1+k\alpha)\frac{g(\lambda^{-1}r)}{\lambda^{-1}r}-g'(\lambda^{-1}r) \leq 0.$$
Therefore,
\begin{equation}\label{8}
    \begin{split}
        \p_t(\frac{1}{2}|\n\phi|^2) \leq & -(\beta - k\alpha - 1)\rho \sigma_k^\alpha |\n \phi|^2 e^{(\beta-1)\phi}\\
        &  +\alpha \sigma_k^{\alpha-1}\dot{\sigma}_k^{pq}e^{\beta\phi}\left(\left(\frac{1}{2}|\n \phi|^2\right)_{pq}-\phi_{ip} \phi_{iq}+\phi_p \phi_q-\delta_{pq}|\n \phi|^2 \right) \\
        &+ \alpha \sigma_k^{\alpha-1}\dot{\sigma}_k^{pq}\lambda^\beta g(\lambda^{-1}r)\left(\left(\frac{1}{2}|\n \phi|^2\right)_{pq}-\phi_{ip} \phi_{iq}+\phi_p \phi_q-\delta_{pq}|\n \phi|^2 \right). 
    \end{split}
\end{equation}
Recall $\beta \geq k\alpha + 1$, we have 
$$-(\beta - k\alpha - 1)\rho \sigma_k^\alpha |\n \phi|^2 e^{(\beta-1)\phi} \leq 0$$
Since $\theta_0$ is the maximum point of $(\frac{1}{2}|\n\phi|^2) $, it follows $\left(\frac{1}{2}|\n \phi|^2\right)_{kl} \leq 0$. Notice that $e^{(k\alpha+1)\phi} + e^{(k\alpha+1)\gamma t}g(e^{-\gamma t}r) = \lambda^\beta f(\lambda^{-1}r) > 0$, we have
$$\alpha \sigma_k^{\alpha-1}\dot{\sigma}_k^{pq} \lambda^\beta f(\lambda^{-1}r)\phi_{ip}\phi_{iq} \geq 0$$
$$\alpha \sigma_k^{\alpha-1}\dot{\sigma}_k^{pq}\lambda^\beta f(\lambda^{-1}r)(\phi_p\phi_q - \delta_{pq}|\n \phi|^2) \leq 0$$
This implies $\p_t |\n \phi|^2 \leq 0$.
\end{proof}
A corollary of the $C^1$ estimate is the following bound on the support function.
\begin{corollary}\label{cor3}
    Along the normalized flow \eqref{2}, we have $u \geq C > 0$.
\end{corollary}
\begin{proof}
   Recall $u=\frac{r}{\rho}=\frac{r}{\sqrt{1+|\n \phi|^2}}$. Since $|\n \phi|$ and $r$ are bounded by $C^0$ and $C^1$ estimates, it follows that $u$ is bounded. 
   
   Recall $u=\langle X, \nu\rangle=r\left\langle\partial_r, \nu\right\rangle$, we have

$$
\left\langle\partial_r, \nu\right\rangle=\frac{u}{r} \geq C > 0
$$

Thus, the hypersurface preserves star-shaped along the flow \eqref{2}.
\end{proof}
\section{$C^2$-estimate}
In this section, we do analysis on the hypersurfaces $M_t$. Let $\n$ be the Levi-Civita connection on $M_t$.
Denote the parabolic operator 
$$L = (\p_t - \lambda^\beta f(\lambda^{-1}r)\alpha \sigma_k^{\alpha-1}\dot{\sigma}_k^{ij}\n_j\n_i).$$

\subsection{Lower bound of $\sigma_k^\alpha$}
\begin{lemma}\label{lem 2}
    Along the normalized flow \eqref{2}, we have 
    $$\Phi \geq C > 0$$
\end{lemma}
\begin{proof}
    By Corollary \ref{cor1}, at the minimum point of $\Phi$ we have
    $$\frac{d}{dt} \Phi_{\min} \geq  (\beta - k\alpha)\gamma \Phi_{\min}- \frac{f'(\lambda^{-1}r)}{\lambda f(\lambda^{-1}r)}\frac{\Phi_{\min}^2 u}{r}$$
    In a view of \eqref{12}, Corollary \ref{cor3} and $C^0$-estimate, we conclude
    $$\frac{f'(\lambda^{-1}r)}{\lambda f(\lambda^{-1}r)}\frac{ u}{r} \leq C$$
    Therefore,
     $$\frac{d}{dt} \Phi_{\min} \geq (\beta-k\alpha)\gamma \Phi_{\min} - C\Phi_{\min}^2$$
     This implies $\Phi_{\min} \geq C > 0$.
\end{proof}
\subsection{Upper Bound of $\sigma_k^\alpha$}
\begin{lemma}\label{lem4}
    Along the normalized flow \eqref{2}, $\Phi$ has an upper bound, i.e., $\Phi \leq C$.
\end{lemma}
\begin{proof}
Consider the auxiliary function $$Q = \log \Phi - \log (u-a)$$
where $a = \frac{1}{2}\inf_{M \times [0,T]} u$. Then
\begin{equation*}
\begin{split}
LQ =&\frac{L\Phi}{\Phi} - \frac{Lu}{u-a} + \frac{1}{\Phi^2}\lambda^\beta f(\lambda^{-1}r)\alpha \sigma_k^{\alpha-1}\dot{\sigma}_k^{ij}\n_j\Phi \n_i\Phi\\
&- \frac{1}{(u-a)^2}\lambda^\beta f(\lambda^{-1}r)\alpha \sigma_k^{\alpha-1}\dot{\sigma}_k^{ij}\n_ju\n_iu.
\end{split}
\end{equation*}
At the maximum point of $Q$, we have $\frac{\n \Phi}{\Phi} = \frac{\n u}{u-a}$. Hence,
\begin{equation}\label{13}
    \begin{split}
        LQ = &\frac{L\Phi}{\Phi} - \frac{Lu}{u-a}\\
        = &\lambda^\beta f(\lambda^{-1}r)\alpha\sigma_k^{\alpha-1}\dot{\sigma}_k^{ij}h_i{ }^kh_{kj} + \gamma(\beta-k\alpha)- \frac{f'(\lambda^{-1}r)}{\lambda f(\lambda^{-1}r)}\frac{\Phi u}{r} \\
        &+(1+k\alpha)\frac{\Phi}{u-a} - \gamma \frac{u}{u-a} - \frac{1}{u-a}\lambda^{\beta-1} f'(\lambda^{-1}r)\sigma_k^\alpha\l X, \n r \r\\
           &-\lambda^\beta f(\lambda^{-1}r)\alpha\sigma_k^{\alpha-1}\dot{\sigma}_k^{ij}h_i{ }^lh_{lj}\frac{u}{u-a}\\
           \leq & -C\sigma_k^{\alpha-1}\dot{\sigma}_k^{ij}h_i{ }^lh_{lj} + C\Phi + C
    \end{split}
\end{equation}
where we used $C^0$ estimate together with \eqref{10},\eqref{11},\eqref{12} and Lemma \ref{lem 2}.

We compute in the normal coordinates such that $\{h_i{ }^j\}$ is diagonal. We then have
$$\dot{\sigma}_k^{ij}h_i{ }^lh_{lj} = \frac{\p \sigma_k}{\p \kappa_i} \kappa_i^2 = H\sigma_k-(k+1)\sigma_{k+1}.$$
Recall Netwon-MacLaurin inequalities $\sigma_{k+1} \leq C^{k+1}_n(\frac{\sigma_k}{C_n^k})^{\frac{k+1}{k}}$ and $H \geq n(\frac{\sigma_k}{C_n^k})^{\frac{1}{k}}$, we have
\begin{equation*}
    \begin{split}
        H\sigma_k-(k+1)\sigma_{k+1} \geq &n(\frac{\sigma_k}{C_n^k})^{\frac{1}{k}}\sigma_k - (k+1)C^{k+1}_n(\frac{\sigma_k}{C_n^k})^{\frac{k+1}{k}}\\
        =& (\frac{n}{(C_n^k)^{\frac{1}{k}}}-(k+1)\frac{C_n^{k+1}}{(C_n^k)^{1+\frac{1}{k}}})\sigma_k^{1 + \frac{1}{k}}.
    \end{split}
\end{equation*}
Notice that $\frac{n}{(C_n^k)^{\frac{1}{k}}}-(k+1)\frac{C_n^{k+1}}{(C_n^k)^{1+\frac{1}{k}}} > 0$. Putting this term into \eqref{13}, we have 
$$LQ \leq -C\sigma_k^{\alpha + \frac{1}{k}} + C \Phi + C \leq -C\Phi^{1+\frac{1}{k\alpha}} + C\Phi + C.$$
This implies $\Phi \leq C$.
\end{proof}
\subsection{Curvature Estimate}
\subsubsection{$k=1$}
We first derive the curvature bounds when $k=1$.

\begin{lemma}
    Along the normalized flow (\ref{3}), the curvature is bounded, i.e.,
    $$|A|^2 \leq C.$$
\end{lemma}
\begin{proof}
Take $k=1$ in \eqref{20}, we have
 \begin{equation*}
\begin{split}
    L|A|^2=& 2\lambda^\beta f(\lambda^{-1}r) \alpha(\alpha - 1)H^{\alpha - 2} \n^jH\n_iH h_j{ }^i \\
        &-2\gamma |A|^2 - 2(\alpha-1) \Phi tr(W^3) +  2\lambda^\beta f(\lambda^{-1}r) \alpha H^{\alpha-1}|A|^4\\
         &+ 4\lambda^{\beta-1}f'(\lambda^{-1}r)\n_ir \n^jH^\alpha h_j{ }^i -2\lambda^\beta f(\lambda^{-1}r)  \alpha H^{\alpha-1}|\n A|^2\\
        & +2\lambda^{\beta-1}f'(\lambda^{-1}r)\n^j\n_ir h_j{ }^i H^\alpha+ 2\lambda^{\beta-2}f''(\lambda^{-1}r)\n^jr\n_ir H^\alpha h_j{ }^i.
    \end{split}
\end{equation*}
Take $k=1$ in Corollary \ref{cor1}, we have
    \begin{equation*}
    \begin{split}
       L \Phi = &\lambda^\beta f(\lambda^{-1}r)\Phi\alpha H^{\alpha-1}|A|^2 
        + (\beta-k\alpha)\gamma \Phi- \frac{f'(\lambda^{-1}r)}{\lambda f(\lambda^{-1}r)}\frac{\Phi^2 u}{r}.
    \end{split}
\end{equation*}
Define an auxiliary function $Q = \log |A|^2 - 2B\log(\Phi - a) $ where $a = \frac{\inf_{M \times [0,T)]} \Phi}{2}$ and $B$ is a constant to be determined, then
\begin{equation*}
    \begin{split}
        LQ = & \frac{L|A|^2}{|A|^2} + \frac{ \lambda^{\beta}f(\lambda^{-1}r)\alpha H^{\alpha-1}}{|A|^4}|\n |A|^2|^2-2B \frac{L\Phi}{\Phi-a} -2B \frac{\lambda^{\beta}f(\lambda^{-1}r) \alpha H^{\alpha-1}}{(\Phi - a)^2} |\n \Phi|^2\\
        =& 2\lambda^{\beta}f(\lambda^{-1}r) (1-\frac{1}{\alpha})H^{-\alpha} \frac{h_j{ }^i}{|A|^2}\n^jH^\alpha\n_iH^\alpha  -2\gamma  - 2(\alpha-1) \Phi \frac{tr(W^3)}{|A|^2}\\
        &+  2\lambda^{\beta}f(\lambda^{-1}r) \alpha H^{\alpha-1}|A|^2+ 4\lambda^{\beta-1}f'(\lambda^{-1}r)\frac{h_j{ }^i}{|A|^2}\n_ir \n^jH^\alpha  \\
        & +2\lambda^{\beta-1}f'(\lambda^{-1}r)\n^j\n_ir \frac{h_j{ }^i}{|A|^2}  H^\alpha+ 2\lambda^{\beta-2}f''(\lambda^{-1}r)\n^jr\n_ir \frac{h_j{ }^i}{|A|^2} H^\alpha \\
         &-2\frac{|\n A|^2}{|A|^2}\lambda^{\beta}f(\lambda^{-1}r) \alpha H^{\alpha-1} + \frac{\lambda^{\beta}f(\lambda^{-1}r) \alpha H^{\alpha-1}}{|A|^4}|\n |A|^2|^2\\
         &-2B\frac{\Phi}{\Phi-a}\lambda^{\beta}f(\lambda^{-1}r)\alpha H^{\alpha-1}|A|^2 
        - 2B\frac{\Phi}{\Phi-a}(\beta-k\alpha)\gamma \\ &+2B\frac{\Phi}{\Phi-a} \frac{f'(\lambda^{-1}r)}{\lambda f(\lambda^{-1}r)}\frac{\Phi u}{r}-2B \frac{\lambda^{\beta}f(\lambda^{-1}r) \alpha H^{\alpha-1}}{(\Phi - a)^2} |\n \Phi|^2
    \end{split}
\end{equation*}
In a view of Lemma \ref{lem4} and $\frac{|tr(W^3)|}{|A|^2} \leq C|A|$, we have
$$- 2(\alpha-1) \Phi \frac{tr(W^3)}{|A|^2} \leq C|A|$$
Notice that $\frac{\Phi}{\Phi - a} \geq \frac{\Phi_{\max}}{\Phi_{\max} - a} = \frac{c}{c-a}$ where $c = \sup _{M \times [0,T)} \Phi$. 

Use Lemmas \ref{lem 2}, \ref{lem4} and \eqref{10}, we have
$$(2-2B\frac{\Phi}{\Phi-a})\lambda^{\beta}f(\lambda^{-1}r)\alpha H^{\alpha-1}|A|^2 \leq -(2 \frac{c}{c-a}B -2) C|A|^2$$ 
if $2 \frac{c}{c-a}B -2 \geq 0$. Similarly, use \eqref{12}, we have
$$2\lambda^{\beta}f(\lambda^{-1}r) (1-\frac{1}{\alpha})H^{-\alpha} \frac{h_j{ }^i}{|A|^2}\n^jH^\alpha\n_iH^\alpha \leq C\frac{|\n H^\alpha|^2}{|A|},$$
$$- 2B\frac{\Phi}{\Phi-a}\gamma  +2B\frac{\Phi}{\Phi-a} \frac{f'(\lambda^{-1}r)}{\lambda f(\lambda^{-1}r)}\frac{\Phi u}{r} \leq C.$$
By standard computations,
$$
\nabla_j \nabla_i r=\frac{\delta_{i j}}{r}-\frac{u}{r} h_{i j}-\frac{\nabla_i r \nabla_j r}{r} .
$$
Hence by Lemma \ref{lem5} and Corollary \ref{cor3},
\begin{equation}\label{17}
    |\n^2 r| \leq C + C|A|.
\end{equation}
Therefore, use \eqref{11},\eqref{15} and Lemma \ref{lem4} we obtain
$$2\lambda^{\beta-1}f'(\lambda^{-1}r)\n^j\n_ir \frac{h_j{ }^i}{|A|^2}  H^\alpha \leq \frac{C}{|A|} + C,$$
$$ 2\lambda^{\beta-2}f''(\lambda^{-1}r)\n^jr\n_ir \frac{h_j{ }^i}{|A|^2} H^\alpha \leq \frac{C}{|A|},$$
$$ 4\lambda^{\beta-1}f'(\lambda^{-1}r)\frac{h_j{ }^i}{|A|^2}\n_ir \n^jH^\alpha  \leq C\frac{|\n H^\alpha|}{|A|}.$$
Taking the above estimates into account, we get
\begin{equation*}
    \begin{split}
        LQ  \leq&-(4B-2) C|A|^2 + C + \frac{C}{|A|}+C\frac{|\n H^\alpha|^2}{|A|}+ C\frac{|\n H^\alpha|}{|A|}  \\
         &+\lambda^{\beta}f(\lambda^{-1}r) \alpha H^{\alpha-1}(-2\frac{|\n A|^2}{|A|^2} + \frac{|\n |A|^2|^2}{|A|^4}-2B \frac{ |\n \Phi|^2}{(\Phi - a)^2}).
    \end{split}
\end{equation*}
By Cauchy's inequality,
$$\frac{\l \n |A|^2,\n \Phi\r}{\Phi - a} = 2(\n_lh_{ij})(h_{ij}\frac{\n_l\Phi}{\Phi-a}) \leq |\n A|^2 +|A|^2\frac{|\n \Phi|^2}{(\Phi - a)^2}.$$
Hence,
$$-2\frac{|\n A|^2}{|A|^2} + \frac{|\n |A|^2|^2}{|A|^4}-2B \frac{ |\n \Phi|^2}{(\Phi - a)^2} \leq  \frac{|\n |A|^2|^2}{|A|^4} - 2\frac{\l \n |A|^2,\n \Phi \r}{(\Phi-a)|A|^2}+(2-2B)\frac{ |\n \Phi|^2}{(\Phi - a)^2}.$$
Since $\frac{\n |A|^2}{|A|^2} = 2B \frac{\n \Phi}{\Phi - a}$ at the maximum point of $Q$, it follows that
\begin{equation*}
       \frac{|\n |A|^2|^2}{|A|^4} - 2\frac{\l \n |A|^2,\n \Phi \r}{(\Phi-a)|A|^2}+(2-2B)\frac{ |\n \Phi|^2}{(\Phi - a)^2}=(4B^2-6B+2)\frac{ |\n \Phi|^2}{(\Phi - a)^2}
\end{equation*}
Take $B \in (1-\frac{a}{c},1)$ such that $$\begin{cases}
    2B^2-3B+1< 0\\
     \frac{c}{c-a}B -1 > 0
\end{cases}$$

Hence, by \eqref{10}, Lemma \ref{lem3} and Lemma \ref{lem4},
\begin{equation*}
        LQ \leq -C|A|^2 + \frac{C}{|A|}+C\frac{|\n H|}{|A|}+C\frac{|\n H|^2}{|A|}+C-C|\n \Phi|^2.
\end{equation*}

Recall $\Phi = \lambda^{\beta}f(\lambda^{-1}r)$, we have
$$\n \Phi =\lambda^{\beta-1}f'(\lambda^{-1}r)\n r H^\alpha +  \lambda^{\beta}f(\lambda^{-1}r) \n H^\alpha.$$
Hence,
\begin{equation*}
    \begin{split}
        |\n \Phi|^2 = &[\lambda^{\beta-1}f'(\lambda^{-1}r)]^2 H^{2\alpha} +  [\lambda^{\beta}f(\lambda^{-1}r)]^2 |\n H^\alpha|^2\\
        &+2\l \lambda^{\beta-1}f'(\lambda^{-1}r)\n r H^\alpha , \lambda^{\beta}f(\lambda^{-1}r) \n H^\alpha\r.
    \end{split}
\end{equation*}
By \eqref{10}, \eqref{11} and Lemma \ref{lem3}, we have 
$$|\n \Phi|^2 \geq  C|\n H|^2-C|\n H| \geq C|\n H|^2 -C$$
where in the last inequality we used $2ab \leq \epsilon^2 a^2 + \frac{1}{\epsilon^2}b^2$.
Therefore,
\begin{equation*}
\begin{split}
     LQ \leq & -C|A|^2  -C|\n H|^2+ \frac{C}{|A|} +C\frac{|\n H|^2}{|A|}+C\frac{|\n H|}{|A|}+C\\
     \leq & -C|A|^2  -C|\n H|^2+ \frac{C}{|A|} +C\frac{|\n H|^2}{|A|}+C.
\end{split}
\end{equation*}
This implies $|A|^2 \leq C$ because if $|A|$ is sufficiently large we have $ LQ \leq 0$ which is a contradiction.
\end{proof}
\subsubsection{$2 \leq k \leq n$}
\begin{lemma}\label{lem6}
    If $\kappa(h_i{ }^j) \in \Gamma_k^+$, then
    $$\ddot{\sigma}_k^{pq,rs} \n_ih_{pq}\n_ih_{rs} \leq -\sigma_k(\frac{\n_i\sigma_k}{\sigma_k}-\frac{\n_iH}{H})((\frac{2-k}{k-1})\frac{\n_i\sigma_k}{\sigma_k}-(\frac{k}{k-1})\frac{\n_iH}{H})$$
\end{lemma}
\begin{proof}
    See \cite{guan2012hypersurfaces}.
\end{proof}
\begin{lemma}
    Along the normalized flow, the curvatures have a uniform bound, i.e.
    $$|A| \leq C$$
\end{lemma}
\begin{proof}
    Since $\kappa(h_i{ }^j) \in \Gamma_k^+$ and $k \geq 2$, it follows $\sigma_2 > 0$ and hence 
    
    $$H^2 = 2\sigma_2 + |A|^2 > |A|^2.$$

    Define $Q = \frac{H}{\Phi - a}$ where $a = \frac{1}{2} \inf_{M \times [0,T)} \Phi$, then
    \begin{equation*}
        LQ = \frac{LH}{\Phi - a} - \frac{H L \Phi}{(\Phi - a)^2} + 2\frac{1}{\Phi - a}\lambda^\beta f(\lambda^{-1}r)\alpha \sigma_k^{\alpha-1}\dot{\sigma}_k^{ij}\n_i \Phi\n_iQ.
    \end{equation*}
    At the maximum point of $Q$ we have $\n Q = 0$. Therefore,
     \begin{equation}\label{16}
     \begin{split}
         LQ = & \frac{LH}{\Phi - a} - \frac{H L \Phi}{(\Phi - a)^2} \\
         =& \lambda^\beta f(\lambda^{-1}r) \alpha(\alpha - 1)\frac{\sigma_k^{\alpha - 2}}{\Phi - a} |\n \sigma_k|^2 +\lambda^\beta f(\lambda^{-1}r) \alpha \frac{\sigma_k^{\alpha - 1}}{\Phi-a}\ddot{\sigma}_k^{pq,st}\n_ih_{pq}\n^ih_{st}\\
         &-\gamma \frac{H}{\Phi-a} - (k\alpha-1) \frac{\Phi}{\Phi-a} |A|^2 +  \lambda^\beta f(\lambda^{-1}r) \alpha \sigma_k^{\alpha-1}\dot{\sigma}_k^{pq}h_{p}{ }^sh_{sq}\frac{H}{\Phi-a}\\
         &+ 2\lambda^{\beta-1} f'(\lambda^{-1}r)\frac{1}{\Phi-a}\l \n r, \n \sigma_k^\alpha \r  +\lambda^{\beta-1} f'(\lambda^{-1}r)\Delta r \frac{\sigma_k^\alpha}{\Phi-a}\\
         &+ \lambda^{\beta-2} f''(\lambda^{-1}r)|\n r|^2\frac{\sigma_k^\alpha}{\Phi-a}-\frac{\Phi}{\Phi-a}\lambda^\beta f(\lambda^{-1}r) \alpha \sigma_k^{\alpha-1}\dot{\sigma}_k^{pq}h_{p}{ }^sh_{sq}\frac{H}{\Phi-a}\\
         &-\gamma(\beta-k\alpha) \frac{H\Phi}{(\Phi-a)^2} + \frac{f'(\lambda^{-1}r)}{\lambda f(\lambda^{-1}r)}\frac{H \Phi^2 u}{r(\Phi-a)^2}.
     \end{split}
    \end{equation}
Since $\n \sigma_k^\alpha = \alpha \sigma_k^{\alpha-1} \n \sigma_k$, we have
$$\lambda^\beta f(\lambda^{-1}r) \alpha(\alpha - 1)\frac{\sigma_k^{\alpha - 2}}{\Phi - a} |\n \sigma_k|^2 = \lambda^\beta f(\lambda^{-1}r) \frac{\alpha - 1}{(\Phi - a)\sigma_k^\alpha \alpha}|\n \sigma_k^\alpha|^2. $$
By Lemma \ref{lem6}, 
\begin{equation*}
    \begin{split}
    \ddot{\sigma}_k^{pq,rs} \n_ih_{pq}\n_ih_{rs} \leq & -\sigma_k(\frac{\n_i\sigma_k}{\sigma_k}-\frac{\n_iH}{H})((\frac{2-k}{k-1})\frac{\n_i\sigma_k}{\sigma_k}-(\frac{k}{k-1})\frac{\n_iH}{H})\\
        =&-\sigma_k[(\frac{2-k}{k-1})\frac{|\n \sigma_k|^2}{\sigma_k^2} - \frac{2}{k-1}\frac{\l \n \sigma_k,\n H \r}{H \sigma_k} + \frac{k}{k-1}\frac{|\n H|^2}{H^2}].
    \end{split}
\end{equation*}
Therefore,
\begin{equation*}
    \begin{split}
        &\lambda^\beta f(\lambda^{-1}r) \alpha \frac{\sigma_k^{\alpha - 1}}{\Phi-a}\ddot{\sigma}_k^{pq,st}\n_ih_{pq}\n^ih_{st}\\
        \leq & -\lambda^\beta f(\lambda^{-1}r) \alpha \frac{\sigma_k^{\alpha }}{\Phi-a}[(\frac{2-k}{k-1})\frac{|\n \sigma_k|^2}{\sigma_k^2} - \frac{2}{k-1}\frac{\l \n \sigma_k,\n H \r}{H \sigma_k} + \frac{k}{k-1}\frac{|\n H|^2}{H^2}]\\
        =&-\lambda^\beta f(\lambda^{-1}r)  \frac{1}{\Phi-a}[(\frac{2-k}{k-1})\frac{|\n \sigma_k^\alpha|^2}{\sigma_k^\alpha \alpha} - \frac{2}{k-1}\frac{\l \n \sigma^\alpha_k,\n H \r}{H} + \frac{k}{k-1}\alpha \sigma_k^\alpha\frac{|\n H|^2}{H^2}].
    \end{split}
\end{equation*}
Hence the first two terms in \eqref{16} can be bounded as follows
\begin{equation*}
    \begin{split}
        &\lambda^\beta f(\lambda^{-1}r) \alpha(\alpha - 1)\frac{\sigma_k^{\alpha - 2}}{\Phi - a} |\n \sigma_k|^2 
        +\lambda^\beta f(\lambda^{-1}r) \alpha \frac{\sigma_k^{\alpha - 1}}{\Phi-a}\ddot{\sigma}_k^{pq,st}\n_ih_{pq}\n^ih_{st}\\
        \leq& \frac{\lambda^\beta f(\lambda^{-1}r)}{\Phi-a}[(1-\frac{1}{\alpha(k-1)})\frac{|\n \sigma_k^\alpha|^2}{\sigma_k^\alpha} + \frac{2}{k-1}\frac{\l\n \sigma_k^\alpha,\n H\r}{H} - \frac{k\alpha}{k-1} \sigma_k^\alpha\frac{|\n H|^2}{H^2}].
    \end{split}
\end{equation*}
Since $\n Q = 0$, it follows $\frac{\n H}{H} = \frac{\n \Phi}{\Phi-a}$. Recall $\Phi = \lambda^\beta f(\lambda^{-1}r) \sigma_k^\alpha$ and hence 
$$\n \Phi = \lambda^\beta f(\lambda^{-1}r) \n\sigma_k^\alpha + \lambda^{\beta-1} f'(\lambda^{-1}r) \sigma_k^\alpha\n r.$$
By \eqref{11}, \eqref{12} we conclude,
\begin{equation*}
    \begin{split}
        &\lambda^\beta f(\lambda^{-1}r) \alpha(\alpha - 1)\frac{\sigma_k^{\alpha - 2}}{\Phi - a} |\n \sigma_k|^2 
        +\lambda^\beta f(\lambda^{-1}r) \alpha \frac{\sigma_k^{\alpha - 1}}{\Phi-a}\ddot{\sigma}_k^{pq,st}\n_ih_{pq}\n^ih_{st}\\
        \leq& \frac{\lambda^\beta f(\lambda^{-1}r)}{\Phi-a}(1-\frac{1}{\alpha(k-1)} + \frac{2}{k-1}\frac{\Phi}{\Phi-a} - \frac{k\alpha}{k-1}\frac{\Phi^2}{(\Phi-a)^2})\frac{|\n \sigma_k^\alpha|^2}{\sigma_k^\alpha} +C |\n \sigma_k^\alpha|+C.
    \end{split}
\end{equation*}
View 
$$ 1-\frac{1}{\alpha(k-1)} + \frac{2}{k-1}\frac{\Phi}{\Phi-a} - \frac{k\alpha}{k-1}\frac{\Phi^2}{(\Phi-a)^2} = p(\frac{\Phi}{\Phi - a})$$ 
as a quadratic function about $\frac{\Phi}{\Phi - a}$. The axis of symmetry is $\frac{1}{k\alpha}$ and $1 <\frac{c}{c-a}\leq  \frac{\Phi}{\Phi - a} \leq 2$. Since $\frac{1}{k\alpha} \leq 1$, it follows
$$p(\frac{\Phi}{\Phi - a})\leq p(\frac{c}{c-a}) < p(1) = \frac{(k\alpha-1)(1-\alpha)}{\alpha(k-1)} \leq 0$$
by our assumption $\alpha \geq 1$ or $\alpha = \frac{1}{k}$.

Therefore,
\begin{equation*}
    \begin{split}
        &\lambda^\beta f(\lambda^{-1}r) \alpha(\alpha - 1)\frac{\sigma_k^{\alpha - 2}}{\Phi - a} |\n \sigma_k|^2 
        +\lambda^\beta f(\lambda^{-1}r) \alpha \frac{\sigma_k^{\alpha - 1}}{\Phi-a}\ddot{\sigma}_k^{pq,st}\n_ih_{pq}\n^ih_{st}\\
        \leq& -C|\n \sigma_k^\alpha|^2 +C |\n \sigma_k^\alpha|+C
    \end{split}
\end{equation*}
By \eqref{11}, \eqref{12}, \eqref{15}, \eqref{17}, Lemmas \ref{lem 2}, \ref{lem4} and $|\n r| = 1$,
$$-\frac{a}{\Phi - a}\lambda^\beta f(\lambda^{-1}r) \alpha \sigma_k^{\alpha-1}\dot{\sigma}_k^{pq}h_p{ }^sh_{sq}\frac{H}{\Phi-a}\leq -C\dot{\sigma}_k^{pq}h_p{ }^sh_{sq}H,$$
$$2\lambda^{\beta-1} f'(\lambda^{-1}r)\frac{1}{\Phi-a}\l \n r, \n \sigma_k^\alpha \r \leq C|\n \sigma_k^\alpha|,$$
$$\lambda^{\beta-1} f'(\lambda^{-1}r)\Delta r \frac{\sigma_k^\alpha}{\Phi-a} \leq C + C|A| \leq C + CH,$$
$$\lambda^{\beta-2} f''(\lambda^{-1}r)|\n r|^2\frac{\sigma_k^\alpha}{\Phi-a} \leq C,$$
$$\frac{f'(\lambda^{-1}r)}{\lambda f(\lambda^{-1}r)}\frac{H \Phi^2 u}{r(\Phi-a)^2} \leq CH.$$
By our assumption on $\alpha$, we have $k\alpha - 1 \geq 0$ and hence $-(k\alpha - 1)|A|^2 \leq 0$.

Putting the above facts into \eqref{16}, we have
\begin{equation*}
        LQ \leq -C\dot{\sigma}_k^{pq}h_p{ }^sh_{sq}H-C|\n \sigma_k^\alpha|^2 + C|\n \sigma_k^\alpha| + CH + C.
\end{equation*}
Suppose the principle curvatures are arranged as $\kappa_1 \geq \cdots \geq \kappa_n$, we have $\sigma_{k-1} (\kappa|\kappa_1)\kappa_1 \geq \frac{k}{n}\sigma_k$ (see \cite{wang2009k}). Hence at a frame such that $\{h_i{ }^j\}$ is diagonal, we have
$$\dot{\sigma}_k^{pq}h_p{ }^sh_{sq} = \sigma_{k-1}(\kappa|\kappa_i)\kappa_i^2 \geq \sigma_{k-1} (\kappa|\kappa_1)\kappa_1^2 \geq \frac{k}{n}\sigma_k \kappa_1 \geq CH.$$
Therefore,
\begin{equation*}
\begin{split}
        LQ \leq & -C\dot{\sigma}_k^{pq}h_p{ }^sh_{sq}H-C|\n \sigma_k^\alpha|^2 + C|\n \sigma_k^\alpha| + CH + C\\
        \leq & -CH^2-C|\n \sigma_k^\alpha|^2 + C|\n \sigma_k^\alpha| + CH + C\\
        \leq & -CH^2 + C.
\end{split}
\end{equation*}
This implies that $Q$ has a uniform upper bound.
\end{proof}
By $C^0$ and $C^1$ estimates, we have 
$$0 < \frac{1}{C} \leq e^{(\beta-1)\phi}\rho + \frac{\rho}{r}\lambda^{\beta}g(\lambda^{-1}r) \leq C.$$
Combine this with $C^2$ estimate, it follows the equation (\ref{3}) is uniformly parabolic.

Hence by the standard parabolic equation theory (see \cite{krylov1987nonlinear}) we derive the long time existence of the equation (\ref{3}).
\section{Exponential convergence to a sphere}
\begin{lemma}
    Let $\phi$ be a solution to the flow (\ref{3}), then
    $$|\n \phi| \leq Ce^{-ct}$$
    for some positive constant $C,c > 0$.
\end{lemma}
\begin{proof}
    Use (\ref{8}) and recall $g \geq 0$, we conclude
    \begin{equation*}
    \begin{split}
        \p_t(\frac{1}{2}|\n\phi|^2) \leq & -(\beta - k\alpha - 1)\sigma_k^\alpha |\n \phi|^2 e^{(\beta-1)\phi}\\
        &  +\alpha \sigma_k^{\alpha-1}\dot{\sigma}_k^{pq}e^{\beta\phi}\left(\left(\frac{1}{2}|\n \phi|^2\right)_{pq}-\phi_{ip} \phi_{iq}+\phi_p \phi_q-\delta_{pq}|\n \phi|^2 \right)
    \end{split}
\end{equation*}
    From here we can use the same argument as \cite{li2022asymptotic}.
\end{proof}
Since $\n r$ converges to $0$ exponentially, it follows 
\begin{equation}\label{21}
    r_{\max}(t) - r_{\min}(t) \leq \pi |\n r| \leq Ce^{-ct}.
\end{equation}
\begin{lemma}\label{lem7}
    If $\beta > 1 + k\alpha$, then under the normalized flow \eqref{2} the sphere $\s^n(r_0)$ will converge to $\s^n(1)$  as $t \to \infty$.
\end{lemma}
\begin{proof}
     Suppose $\s^n(r(t))$ solves the normalized flow \eqref{2}, then $r$ must satisfy the following ODE
\begin{equation*}
    \frac{d}{dt}r = -\gamma r^{\beta-k\alpha} - \gamma \lambda^\beta g(\lambda^{-1}r)\frac{1}{r^{k\alpha}} + \gamma r.
\end{equation*}
Since $g \geq 0$, $g(0) = \cdots = g^{([\beta])}(0) = 0$ and $r \leq C < \infty$ by the $C^0$ estimate, it follows that
$$0 \leq \gamma \lambda^\beta g(\lambda^{-1}r)\frac{1}{r^{k\alpha}} \leq \gamma C \lambda^{\beta-[\beta]-1}r^{[\beta]+1-k\alpha} \leq C\lambda^{\beta-[\beta]-1}r^{\beta-k\alpha}.$$
Therefore,
$$-(\gamma + C\lambda^{\beta-[\beta]-1})r^{\beta-k\alpha} + \gamma r \leq \frac{d r}{dt} \leq -\gamma r^{\beta-k\alpha}  + \gamma r.$$
Let $r_1(t)$ and $r_2(t)$ be the solutions to the following equations respectively
\begin{equation*}
    \begin{cases}
        \frac{d r_1}{dt} = -(\gamma + C\lambda^{\beta-[\beta]-1})r_1^{\beta-k\alpha} + \gamma r_1,\\
        r_1(0) = r_0,
    \end{cases}
\end{equation*}
\begin{equation*}
    \begin{cases}
        \frac{d r_2}{dt} = -\gamma r_2^{\beta-k\alpha}  + \gamma r_2,\\
        r_2(0) = r_0.
    \end{cases}
\end{equation*}
Recall $\lambda(t) = e^{\gamma t}$. Then $r_1(t)$ and $r_2(t)$ are given by
\begin{equation*}
r_1(t) = 
\begin{cases}
\begin{aligned}
&\Bigg[ \bigg( r_0^{1+k\alpha-\beta} - 1 - \frac{C(1+k\alpha-\beta)}{(\beta-[\beta]-1)\gamma - (1+k\alpha-\beta)\gamma} \bigg) e^{(1+k\alpha-\beta)\gamma t} \\
&\quad + 1 + \frac{C(1+k\alpha-\beta)}{(\beta-[\beta]-1)\gamma - (1+k\alpha-\beta)\gamma} e^{(\beta-[\beta]-1)\gamma t} \Bigg]^{\frac{1}{1+k\alpha-\beta}}, \\
&\qquad \text{if } \beta-[\beta]-1 \neq 1 + k\alpha - \beta, \\[2ex]
&\big[ (r_0^{1+k\alpha-\beta} - 1 - (1+k\alpha-\beta)\gamma t ) e^{(1+k\alpha-\beta)\gamma t} + 1 \big]^{\frac{1}{1+k\alpha - \beta}}, \\
&\qquad \text{if } \beta-[\beta]-1 = 1 + k\alpha - \beta,
\end{aligned}
\end{cases}
\end{equation*}
\begin{equation*}
    r_2(t) = [1+(r_0^{1+k\alpha-\beta}-1)e^{(1+k\alpha-\beta)\gamma t}]^{\frac{1}{1+k\alpha - \beta}} 
\end{equation*}
Notice that $r_i(t) \to 1$ as $t \to \infty$ for $i = 1,2$. By ODE comparison, we have $r_1(t) \leq r(t) \leq r_2(t)$ and hence $r(t) \to 1$ as $t \to \infty$.
\end{proof}
When $\beta > 1 + k\alpha$, let $\s^n(r_1)$ and $\s^n(r_2)$ be two spheres such that $\s^n(r_1) \subset \Omega_0 \subset \s^n(r_2)$ where $\Omega_0$ is the domain bounded by the initial hypersurface $M_0$. By comparison principle and Lemma \ref{lem7}, we conclude $r(t) \to 1$ as $t \to \infty$. By \eqref{21}, $r(t)$ converges $1$ exponentially.

When $\beta = 1+k\alpha$, recall in the proof of $C^0$ estimate we show that $r_{\max}(t)$ is decreasing and hence we can assume $r_{\max}(t) \to r_0$ as $t \to \infty$. Therefore, by \eqref{21} we have $r_{\min}(t) \to r_0$. Moreover,
$$\max\{r_{\max}(t) - r_0,r_0 - r_{\min}(t)\} \leq r_{\max}(t)-r_{\min}(t) \leq Ce^{-ct}.$$
This implies $r(\cdot,t) \to r_0$ exponentially fast as $t \to \infty$.

We now prove the exponential decay of $C^k$ norms for all $k \geq 1$. In both cases $\beta > 1 + k\alpha$ and $\beta = 1 + k\alpha$, we already show that the function $r(t)$ converges to a limit denoted by $r_0$ as $t \to \infty$. Recall the interpolation inequality(see \cite{hamilton1982three})
\begin{equation}
    \int_{\s^n} |\n^m T|^2 \leq C_{m,n}(\int_{\s^n} |\n^l T|^2 )^{\frac{m}{l}}(\int_{\s^n} |T|^2 )^{1-\frac{m}{l}} 
\end{equation}
where $T$ is any smooth tensor field on $\s^n$ and $l,m$ are integers such that $0 \leq m\leq l$. Applying this to $r - r_0$ and recall all derivatives of $r$ are bounded independent of $t$, we have
$$\int_{\s^n} |\n^m r| \leq C_m e^{-c_m t}$$
By the Sobolev embedding theorem on $\s^n$(see \cite{aubin2012nonlinear}), we have
$$\Vert r-r_0 \Vert_{C^m(\s^n)} \leq C_{m,l}(\int_{\s^n} |\n^l r|^2 + \int_{\s^n} |r-r_0|^2)^{\frac{1}{2}}$$
for any $l> m + \frac{n}{2}$. Therefore, $\Vert r - r_0 \Vert_{C^m(\s^n)} \to 0$ exponentially fast as $t \to \infty$ and hence $r(\cdot,t)$ converges to $r_0$ in the $C^\infty$ topology.

{\bf Acknowledgements.}
Research of the first author was partially supported by National Key R$\&$D Program of China (No. 2022YFA1005500) and Natural Science Foundation of China under Grant No. 12031017.


\vspace{5mm}


\begin{thebibliography}{20}

\bibitem{andrews1994contraction}
Ben Andrews. Contraction of convex hypersurfaces in Euclidean space. Calculus of Variations and Partial Differential Equations 2.2 (1994): 151-171.

\bibitem{andrews1999gauss}
Ben Andrews. Gauss curvature flow: the fate of the rolling stones. Inventiones mathematicae 138.1 (1999): 151-161.

\bibitem{andrews2016flow}
Ben Andrews, Pengfei Guan, Lei Ni. Flow by powers of the Gauss curvature. Advances in Mathematics 299 (2016): 174-201.

\bibitem{aubin2012nonlinear}
Thierry Aubin. Nonlinear Analysis on Manifolds. Monge–Amp\`ere Equations. Springer, New York
 (1982)

\bibitem{brendle2017asymptotic}
Simon Brendle, Kyeongsu Choi, Panagiota Daskalopoulos. Asymptotic behavior of flows by powers of the Gaussian curvature. (2017): 1-16.

\bibitem{chow1985deforming}
Bennett Chow. Deforming convex hypersurfaces by the $ n$-th root of the Gaussian curvature. Journal of Differential Geometry 22.1 (1985): 117-138.

\bibitem{chow1987deforming}
Bennett Chow. Deforming convex hypersurfaces by the square root of the scalar curvature. Inventiones mathematicae 87.1 (1987): 63-82.

\bibitem{firey1974shapes}
William J. Firey. Shapes of worn stones. Mathematika 21.1 (1974): 1-11.

\bibitem{guan2012hypersurfaces}
Pengfei Guan, Junfang Li, Yanyan Li. Hypersurfaces of prescribed curvature measure. Duke Math.J. 161 (10) (2012): 1927-1942.

\bibitem{guan2009existence}
Pengfei Guan, Changshou Lin, Xi-Nan Ma. The existence of convex body with prescribed curvature measures. International Mathematics Research Notices 2009.11 (2009): 1947-1975.

\bibitem{guan2017entropy}
Pengfei Guan and Lei Ni. Entropy and a convergence theorem for Gauss curvature flow in high dimension. Journal of the European Mathematical Society 19.12 (2017): 3735-3761.

\bibitem{hamilton1982three}
Richard S. Hamilton. Three-manifolds with positive Ricci curvature. Journal of Differential geometry 17.2 (1982): 255-306.

\bibitem{huang2016geometric}
Yong Huang, Erwin Lutwak, Deane Yang, Gaoyong Zhang. Geometric measures in the dual Brunn–Minkowski theory and their associated Minkowski problems. Acta Math. (2016): 325-388.

\bibitem{huisken1984flow}
Gerhard Huisken. Flow by mean curvature of convex surfaces into spheres. Journal of Differential Geometry 20.1 (1984): 237-266.

\bibitem{li2022asymptotic}
Haizhong Li, Botong Xu, Ruijia Zhang. Asymptotic convergence for a class of anisotropic curvature flows. Journal of Functional Analysis 282.12 (2022): 109460.

\bibitem{li2020asymptotic}
Qirui Li, Weimin Sheng, Xujia Wang. Asymptotic convergence for a class of fully nonlinear curvature flows. Journal of Geometric Analysis 30.1 (2020): 834-860.

\bibitem{li2020flow}
Qirui Li, Weimin Sheng, Xujia Wang. Flow by Gauss curvature to the Aleksandrov and dual Minkowski problems. Journal of the European Mathematical Society (EMS Publishing) 22.3 (2020).

\bibitem{lutwak2018lp}
Erwin Lutwak, Deane Yang, Gaoyong Zhang. Lp dual curvature measures. Advances in Mathematics 329 (2018): 85-132.

\bibitem{krylov1987nonlinear}
N. V. Krylov. Nonlinear elliptic and parabolic equations of the second order. Dordrecht, Holland: D. Reidel Publishing Company, 1987.

\bibitem{wang2009k}
Xujia Wang. The k-Hessian equation. Geometric analysis and PDEs. Berlin, Heidelberg: Springer Berlin Heidelberg, 2009. 177-252.

\end{thebibliography}
\end{document}